\documentclass[12pt]{amsart}
\usepackage{fullpage}
\usepackage{mathtools}
\usepackage{amssymb,amsmath,amsthm,amscd,mathrsfs,graphicx}
\usepackage[dvipsnames]{xcolor}
\usepackage{bm}
\usepackage{dsfont}
\usepackage{enumerate}
\usepackage{epsfig}
\usepackage{float}
\usepackage{latexsym, amsxtra}
\usepackage{pdfpages}
\usepackage{multicol}
\usepackage[normalem]{ulem}
\usepackage{psfrag}
\usepackage{stmaryrd}
\usepackage{tikz}
\usetikzlibrary{positioning, shapes.geometric}
\usepackage[T1]{fontenc}
\usepackage{url}
\usepackage{verbatim}
\usepackage{xy}
\usepackage{indentfirst}
\usepackage{tikz-cd}
\tikzset{%
    symbol/.style={%
        ,draw=none
        ,every to/.append style={%
            edge node={node [sloped, allow upside down, auto=false]{$#1$}}}
    }
}
\usepackage{longtable}
\usepackage{multirow}
\usepackage{bbding}
\usepackage{booktabs}

\usepackage[
  colorlinks=true,
  linkcolor=red,
  urlcolor=blue,
  citecolor=blue
]{hyperref}

\usepackage[title]{appendix}

\usepackage{algorithm}
\usepackage{algpseudocode}

\makeatletter
\def\thm@space@setup{%
  \thm@preskip=2ex \thm@postskip=2ex
}
\makeatother

\makeatletter

\newcommand{\Rmnum}[1]{\expandafter\@slowromancap\romannumeral #1@}
\makeatother

\newcommand{\id}{\mathrm{id}}
\numberwithin{equation}{section}
\theoremstyle{plain}

\newtheorem{thm}{Theorem~}[section] 
\newtheorem{lem}[thm]{Lemma~}

\newtheorem{fact}[thm]{Fact}
\newtheorem{prop}[thm]{Proposition~}

\newtheorem{cor}[thm]{Corollary~}
\newtheorem{claim}[thm]{Claim~}
\newtheorem{prop-def}[thm]{Proposition-Definition~}

\theoremstyle{remark}
\newtheorem{rmk}[thm]{Remark~}

\theoremstyle{definition}
\newtheorem{defn}[thm]{Definition~}

\newcommand{\CC}{\mathbb{C}}
\newcommand{\ZZ}{\mathbb{Z}}

\newcommand{\LL}{\mathbb{L}}
\newcommand{\PP}{\mathbb{P}}
\newcommand{\FF}{\mathbb{F}}
\newcommand{\QQ}{\mathbb{Q}}

\newcommand{\DD}{\mathbb{D}}

\newcommand\PGL{\mathrm{PGL}}

\newcommand\PSL{\mathrm{PSL}}

\newcommand\Hom{\mathrm{Hom}}

\newcommand\rank{\mathrm{rank}}

\newcommand\SL{\mathrm{SL}}

\newcommand\diag{\mathrm{diag}}
\newcommand\GL{\mathrm{GL}}
\newcommand\Sym{\mathrm{Sym}}

\newcommand\Span{\mathrm{Span}}

\newcommand\ord{\mathrm{ord}}
\newcommand\QD{\mathrm{QD}}

\newcommand{\Aut}{\mathrm{Aut}}

\newcommand{\bs}{\backslash}
\newcommand{\dbs}{\bs\hspace{-0.5mm}\bs}

\title{Non-symplectic Indices of Automorphism Groups of Smooth Cubic Fourfolds}

\vspace{1.2cm}
\author{Jie Fu, Shihao Wang, Zhiwei Zheng}
\date{}

\newcommand{\Addresses}{{%
		\bigskip
		\footnotesize

        J.~Fu, \textsc{Qiuzhen college, Tsinghua University, Beijing, China}\par\nopagebreak
		\textit{Email address}: \texttt{fu-j21@mails.tsinghua.edu.cn}

		\medskip

        S.~Wang, \textsc{Qiuzhen college, Tsinghua University, Beijing, China}\par\nopagebreak
		\textit{Email address}: \texttt{wangshihao25@mails.tsinghua.edu.cn}

		\medskip
		
		Z.~Zheng, \textsc{Yau Mathematical Sciences Center, Tsinghua University, Beijing, China}\par\nopagebreak
		\textit{Email address}: \texttt{zhengzhiwei@mail.tsinghua.edu.cn}
}}

\begin{document}
\bibliographystyle{amsalpha}

\begin{abstract}
 We study the full automorphism groups of smooth cubic fourfolds with prescribed symplectic automorphism group. Our starting point is the classification of symplectic automorphism groups by Laza and Zheng. We focus on the non-symplectic index, namely, the index of the symplectic automorphism group in the full automorphism group. We prove general restrictions on this index. We also compute bounds by group-theoretic and lattice-theoretic methods. In several cases, we determine all possible indices. For coinvariant lattices of rank $19$, we classify all possible pairs consisting of the symplectic automorphism group and the full automorphism group.
\end{abstract}

\maketitle

 \setcounter{tocdepth}{1}
	\tableofcontents

\emph{Notation}:
\begin{enumerate}
\item For two groups $G_1, G_2$, we write $G_1<G_2$ if $G_1$ is a subgroup of $G_2$.
\item For groups $G_{1}<G_{2}$, we denote by $C_{G_{2}}(G_{1})$ and $N_{G_{2}}(G_{1})$ the centralizer and normalizer of $G_{1}$ in $G_{2}$.
\item For groups $G_{1}$ and $G_{2}$, $G_1:G_2$ is a semidirect product of $G_1$ and $G_2$; $G_{1}.G_{2}$ is an extension group given by some exact sequence
\begin{align*}
    1\longrightarrow G_{1} \longrightarrow G_{1}.G_{2}\longrightarrow G_{2}\longrightarrow 1.
\end{align*}
\item We use $L_2(q)$ to denote $\PSL_2(\FF_q)$. 
\item We denote by $\mu_{n}$ the group of $n$th roots of unity in $\CC^{\times}$. Sometimes we simply use $n$ to denote the cyclic group of order $n$.
\item We use $3^{1+4}$ to denote $3^{1+4}_+$, the extraspecial $3$-group of order $243$ and exponent $3$.
\item For $\frac{1}{n}(m_{1},\dots,m_{k})$, we mean the diagonal matrix $\diag(\zeta_{n}^{m_1},\dots,\zeta_{n}^{m_k})\in \GL(k,\CC)$, where $\zeta_{n}=\exp\left(\frac{2\pi \sqrt{-1}}{n}\right)$ and $m_{1},\dots,m_{k},n\in \mathbb{Z}$. Sometimes we also refer to the image in $\PGL(k,\CC)$.
\item We denote by $\omega$ the primitive cube root of unity $e^{\frac{2\pi \sqrt{-1}}{3}}$.
\item For $g\in \GL(n,\CC)$, denote by $\overline{g}$ the image of $g$ in $\PGL(n,\CC)$ under the natural projection $\GL(n,\CC)\rightarrow \PGL(n,\CC)$.
\item For a permutation $\sigma \in S_n$, we also use $\sigma$ to denote the matrix $A_\sigma \in \GL(n,\CC)$ such that $(e_{\sigma(1)},\ldots,e_{\sigma(n)}) = (e_1,\ldots,e_n)A_\sigma$. For example, we use $(123) \in S_3$ to denote
\[
\begin{pmatrix}
    0&0&1\\
    1&0&0\\
    0&1&0
\end{pmatrix}.
\]
\item We denote by $a^b c$ the rank $2$ quadratic form
$\left(
\begin{smallmatrix}
a & b\\
b & c
\end{smallmatrix}
\right)$.
We write $-(a^b c)\colon=(-a)^{(-b)}(-c)$.
\item Let $A\in \GL(n,\CC)$ and $F(x_1,\ldots,x_n)$ be the defining equation of a hypersurface $X$ in $\PP^{n-1}$. Then $A\cdot F(x_1,\ldots,x_n)\coloneqq F((x_1,\ldots,x_n)A)$ is the defining equation of $A\cdot X$.
\end{enumerate}

\section{Introduction}
For a smooth cubic fourfold $X$, a fundamental aspect of understanding its geometric properties is to study $\Aut(X)$, its group of automorphisms, and $\Aut^s(X)$, its group of symplectic automorphisms. Based on the global Torelli theorem and lattice theory, \cite{laza2022automorphisms} provides a complete classification of possible symplectic automorphism groups $G$ of smooth cubic fourfolds $X$, revealing relations with certain coinvariant lattices $S=S_G(X)$. On the other hand, \cite{yang2024automorphism} classifies all groups that admit faithful actions on smooth cubic fourfolds.

A natural question arises: for a given symplectic automorphism group $G$, what could the full automorphism group be? This is partly measured by the integer $[\Aut(X):\Aut^s(X)]$, which we call the ``non-symplectic index'' for $X$. Determining possible indices $m$ for a given $G$ is the central problem addressed in this paper.

This work aims to systematically investigate the non-symplectic indices of cubic fourfolds. We employ a combination of techniques, including the lattice-theoretic methods established by \cite{Nikulin_1980}, computational group theory via GAP \cite{GAP4} using classification results established in \cite{laza2022automorphisms} and \cite{yang2024automorphism}, geometric-invariant calculations using expressions given in \cite{KOIKE202512} (see also \cite{KOIKE2026} for Corrigendum) and methods introduced in \cite{he2025cubicfourfoldsorder7automorphism}, and arguments on moduli spaces mainly depending on \cite{yu2020moduli}, to identify and constrain the possible non-symplectic indices. The cases $\rank(S)=20$, which correspond to maximal symplectic automorphism groups and single points in the moduli, are already solved in \cite{laza2022automorphisms}.

As the main result of this paper, we give several restrictions of possible indices for a given isomorphism class of $\Aut^s(X)$. That is,

\begin{thm}\label{thm:main}
    For some groups $G$, we find all possible non-symplectic indices $[\Aut(X):\Aut^s(X)]$ for smooth cubic fourfolds $X$ with $\Aut^s(X)\cong G$, as listed in Table~\ref{table: main}. Precisely, the nonempty items in the column ``all possible indices'' consist of all possible non-symplectic indices.
\end{thm}

In the cases where the coinvariant lattice $S_G(X)$ has rank 19, we provide a complete classification of the groups $\Aut(X)$:

\begin{thm} 

\label{thm:19_class}
    Suppose that $X$ is a smooth cubic fourfold with
$\rank(S_{\Aut^s(X)}(X))=19$. Then all possible pairs
$(\Aut^s(X),\Aut(X))$ are listed in Table~\ref{table:19class}. 
    \begin{longtable}{|p{.24\textwidth} | c | p{.24\textwidth} | c |}
    \hline
    $(\Aut^s(X),\Aut(X))$ & $\mathrm{Index}$ & $(\Aut^s(X),\Aut(X))$ & $\mathrm{Index}$\\
    \hline
    $(3^{1+4}:2.2, 3^{1+4}:D_8)$ & $2$  & $(S_5,S_5\times 2)$ & $2$ \\
    \hline
    $(A_6,A_6)$ & $1$ & $(M_9,M_9)$ & $1$\\
    \hline
    $(A_6,S_6)$ & $2$ & $(M_9,M_9:3)$ & $3$\\
    \hline
    $(L_2(7),L_2(7))$ & $1$ & $(N_{72},N_{72}\times 2)$ & $2$ \\
    \hline
    $(L_2(7),L_2(7):2)$ & $2$ & $(T_{48},T_{48})$ & $1$\\
    \hline
    $(S_5,S_5)$ & $1$ &  &\\
    \hline
    \caption{Rank $19$ Classification}
    \label{table:19class}
    \end{longtable}
\end{thm}

For details, see \S~\ref{sect:rk19}, which finishes with the proof of this theorem.

We also gave some general results to restrain non-symplectic indices and treat several other cases, and some of them are used to show Theorems~\ref{thm:main} and \ref{thm:19_class}:
\begin{itemize}
    \item The non-symplectic index must be a power of two or three times a power of two (see Proposition~\ref{prop:2k3l});
    \item If index $m$ could be realized in some moduli, many of $m$'s divisors could be realized in the moduli as well (see Corollaries \ref{cor:reduce} and \ref{cor:reduce_refined});
    \item Each positive-dimensional irreducible moduli component has a ``generic non-symplectic index.'' This index is either $1$ or $2$ and is determined by the associated lattice data (see Proposition-Definitions \ref{prop-def:generic1}, \ref{prop-def:generic2} and Corollary~\ref{cor:gen12});
    \item For a fixed symplectic automorphism group, the index must divide one of the group-theoretic bounds in Proposition~\ref{prop:YYZ_bounds} and Table~\ref{table: main}.
    \item If the symplectic automorphism group is trivial, then all possible automorphism groups are those cyclic groups of order dividing $32$ or $48$ (see Proposition~\ref{prop:sym_trivial});
\end{itemize}

The paper is organized as follows. Section~\ref{sect:prel} recalls the lattice structure of the integral middle cohomology of a cubic fourfold, the global Torelli theorem, and the classification of symplectic automorphism groups and their moduli spaces; it also records a general result on families of invariant hypersurfaces. Section~\ref{sect:latt} develops the lattice- and moduli-theoretic tools used to analyze non-symplectic indices. Section~\ref{sect:indice} establishes restrictions on these indices and several existence results. Section~\ref{sect:calc} describes the group-theoretic computation, the GIT method and its detectability criterion, and two elementary tools involving differentials and Hessian matrices. Section~\ref{sect:rk19} treats the rank-$19$ cases in detail, using both GIT and lattice theory. Section~\ref{sect:small_rank} treats the case of trivial symplectic automorphism group. Finally, Section~\ref{sect:main_table} presents the main table and completes the proof of Theorem~\ref{thm:main}.

\textbf{Acknowledgements:} We thank Xun Yu for his interest and helpful comments. We thank Xuancong He, Zhirui Li, Shiyixin Liang and Zichang Wang for helpful discussions. The third author is partially supported by NSFC 12301058.

\section{Preliminary}\label{sect:prel}

\subsection{Hodge Structure and Global Torelli Theorem}

Let $X\subset\PP^5$ be a smooth cubic fourfold. It is well-known that the lattice $\Lambda\coloneqq H^4(X,\ZZ)\cong I_{21,2}$ is odd and unimodular. Let $h\in H^2(X, \ZZ)$ be the hyperplane class, and denote by $\Lambda_0$ the orthogonal complement of $h^2$ in $\Lambda$. Then $\langle h^2\rangle \cong \langle 3\rangle$ and $\Lambda_0\cong E_8^{\oplus 2}\oplus U^{\oplus 2}\oplus A_2$.

Moreover, one has Hodge numbers $h^{4,0}=h^{0,4}=0$, $h^{3,1}=h^{1,3}=1$, and $h^{2,2}=21$. Recall the global Torelli theorem for cubic fourfolds, originally shown in \cite{voisin1986torelli,voisin2008erratum}, with an alternative proof provided by \cite{looi2009period}.
\begin{thm}[Global Torelli theorem for cubic fourfolds]
\label{theorem: global torelli}
    Let $X$ and $X^{\prime}$ be two smooth cubic fourfolds. Any Hodge isometry $H^4(X^{\prime},\ZZ)\xrightarrow{\sim} H^4(X,\ZZ)$ preserving the square of the hyperplane class is induced by a unique isomorphism $X\xrightarrow{\sim} X^{\prime}$.
\end{thm}

The automorphism group of a cubic fourfold $X$ is denoted by $\Aut(X)$, and the symplectic automorphism group of $X$, consisting of automorphisms acting trivially on $H^{3,1}(X,\CC)$, is denoted by $\Aut^s(X)$. The general result \cite{Matsumura1963OnTA} implies that any automorphism of a cubic fourfold is linear and that $\Aut(X)$ is finite.

Applying Theorem~\ref{theorem: global torelli} to the case $X^{\prime}=X$, one obtains (see also \cite[Proposition 1.3]{zheng2021orbifold} for more details)
\begin{lem}
    Let $X$ be any smooth cubic fourfold. Then there is a canonical isomorphism
    \[
    \Aut(X)\cong O_{HS}(\Lambda, h^2),
    \]
    where the latter group consists of automorphisms of the lattice $\Lambda$ preserving the Hodge decomposition and $h^2$.
\end{lem}

\subsection{Lattice Theory}

In this section we briefly recall some lattice-theoretic results. For more details see \cite{Nikulin_1980}.

By a lattice we mean a free $\ZZ$-module of finite rank equipped with an integer-valued non-degenerate symmetric bilinear form. For a lattice $L$, its discriminant group, which is a finite abelian group, is defined to be $A_L\coloneqq L^\vee/L$, where $L^\vee$ is the dual lattice of $L$ defined to be $\{x\in L\otimes\QQ\mid\langle x,y\rangle\in\ZZ,\ \forall y\in L\}$. The length $\ell(A_L)$ of $A_L$ is defined as the minimal number of generators of $A_L$ as a group. When $L$ is even (i.e. $\langle x,x\rangle\in 2\ZZ$ for all $x\in L$), we could define the discriminant quadratic form $q_L$ on $A_L$ by $q_L(x+L)=\langle x,x\rangle\bmod 2\ZZ$. Discriminant quadratic forms can be expressed by the Conway--Sloane symbol.

Given a lattice $L$ and a group $G$, for a $G$-action on $L$ we mean a linear action which preserves the bilinear form. This induces a $G$-action on $A_L$ which preserves $q_L$ if $L$ is even. The invariant lattice under the $G$-action is defined as $L^G\coloneqq\{x\in L\mid gx=x,\ \forall g\in G\}$, and the coinvariant lattice is defined to be $S_G(L)\coloneqq(L^G)^\perp$, i.e. the orthogonal complement of $L^G$ in $L$.

Recall the definition of a Leech pair.
\begin{defn}
    We say a pair $(G,S)$, where $G$ is a finite group acting faithfully on a lattice $S$, is a \emph{Leech pair} if $S$ is even, positive definite, and rootless (i.e. $\langle x,x\rangle\neq2$ for all $x\in S$), $G$ acts trivially on $A_S$, and $S^G=0$.
\end{defn}

Recall the Leech lattice $\LL$, which is the unique even positive definite rootless lattice of rank $24$ up to isomorphism, with isometry group isomorphic to the Conway group $Co_0$. The following result has been widely applied, while a complete proof can be found in \cite{zheng2025lemmaleechlikelattices}.
\begin{lem}
    Assume $(G,S)$ is a Leech pair such that $\ell(A_S)+\rank (S)\leq 24$. Then there is a primitive embedding of $S$ into $\LL$ (i.e. $\LL/S$ is a free $\ZZ$-module) extending the $G$-action, which is trivial on the orthogonal complement of $S$ in $\LL$.
\end{lem}

Note that for any subgroup $G<Co_0$, $(G,S_G(\LL))$ is a Leech pair. We say such a pair is \emph{saturated} if $G$ exactly consists of all isometries of $S_G(\LL)$ which act trivially on $A_{S_G(\LL)}$. In \cite{hohn2016290} a full classification of saturated Leech pairs is established.

\subsection{Symplectic Automorphism Groups and Moduli}\label{subsect:ST}

Suppose $X$ is a smooth cubic fourfold and $G<\Aut(X)$. Let $\Lambda_0^G$ be the $G$-invariant sublattice of $\Lambda_0$, which is usually denoted by $T$; let $S_G(X)=(\Lambda_0^G)_{\Lambda_0}^\perp$ be the orthogonal complement of $\Lambda_0^G$ in $\Lambda_0$, which is said to be the $G$-coinvariant lattice and denoted by $S$. Both $S$ and $T$ are primitive sublattices of $\Lambda_0\subset \Lambda$.

When $G<\Aut^s(X)$, $(G,S)$ must be a so-called Leech pair. Using the classification of saturated Leech pairs in \cite{hohn2016290}, \cite{laza2022automorphisms} provides a complete classification of the possible pairs $(G=\Aut^s(X),S)$ and the $G$-action on $S$. Each pair is saturated and determines moduli spaces of smooth cubic fourfolds with symplectic $G$-actions. These spaces need not be irreducible, and every irreducible component whose generic symplectic automorphism group is $G$ has dimension $20-\rank(S)$. There are $6$ maximal symplectic automorphism groups in the case $\rank(S)=20$. See also \cite{KOIKE202512,KOIKE2026}.

This dimension formula is a special case of \cite[Theorem 1.1(i)]{yu2020moduli} for a finite subgroup $G<\PGL(6,\CC)$ with fixed characters on the defining equation and on $H^{3,1}(X,\CC)$. We will use the general formula in Section~\ref{sect:latt}.

\subsection{Actions on Equations}

Assume $n,d\geq 3$, and let $\mathcal{G}$ be a subgroup of $\GL(n,\CC)$. Put
\[
P_{n,d}=\Sym^d((\CC^n)^\vee),
\]
and let $P_{n,d}^{\mathrm{sm}}$ be its open subset of forms defining smooth
hypersurfaces in $\PP^{n-1}$. We write
$P_{n,d}^{\mathcal{G},\mathrm{sm}}$ for the intersection of this open set
with the $\mathcal{G}$-fixed subspace. Here
\[
\Aut(F)=\{g\in\GL(n,\CC)\mid g(F)=F\}.
\]
If $P_{n,d}^{\mathcal{G},\mathrm{sm}}$ is nonempty, the classical result
\cite{Matsumura1963OnTA} and the exact sequence
\[
1\longrightarrow\langle\zeta_d I\rangle\longrightarrow\Aut(F)
\longrightarrow\Aut(X_F)\longrightarrow 1
\]
show that $\mathcal{G}$ is finite. In this case
$P_{n,d}^{\mathcal{G},\mathrm{sm}}$ is irreducible.

The following natural result is important in enumerating families with certain actions.

\begin{prop}\label{prop:GL_family}
Let $\mathcal{G}$ be a finite subgroup of $\GL(n,\CC)$ such that
$P_{n,d}^{\mathcal{G},\mathrm{sm}}$ is nonempty. Then the following are
equivalent:
\begin{enumerate}
    \item For generic $F\in P_{n,d}^{\mathcal{G},\mathrm{sm}}$,
    $\mathcal{G}$ is a proper subgroup of $\Aut(F)$.
    \item For any $F\in P_{n,d}^{\mathcal{G},\mathrm{sm}}$,
    $\mathcal{G}$ is a proper subgroup of $\Aut(F)$.
    \item There is a finite subgroup $\mathcal{G}'<\GL(n,\CC)$ with
    $\mathcal{G}\subsetneq\mathcal{G}'$ such that every
    $F\in P_{n,d}^{\mathcal{G},\mathrm{sm}}$ is linearly equivalent to
    some element of $P_{n,d}^{\mathcal{G}',\mathrm{sm}}$.
\end{enumerate}
\end{prop}

\begin{proof}
We only have to show $(1)\Rightarrow(3)$. Put
$P=P_{n,d}^{\mathcal{G},\mathrm{sm}}$.

Consider the relative automorphism scheme
\[
\mathcal{A}=\{(F,g)\in P\times\GL(n,\CC)\mid g(F)=F\},
\]
with the natural projection $\pi:\mathcal{A}\to P$. For every $F\in P$,
the fiber of $\pi$ is $\Aut(F)$. By Matsumura's theorem
\cite{Matsumura1963OnTA}, $\Aut(F)$ is finite. Hence $\pi$ is quasi-finite.

By assumption, there is a nonempty Zariski-open subset $U\subset P$ such
that $\mathcal{G}\subsetneq\Aut(F)$ for every $F\in U$. Since $P$ is
irreducible, $U$ can be chosen to be irreducible and thus connected. Shrinking
$U$ further if necessary, we may assume that
\[
\pi_U:\mathcal{A}_U\coloneqq\mathcal{A}\times_P U\longrightarrow U
\]
is finite \'{e}tale. Indeed, one first removes the images of those irreducible
components of $\mathcal{A}$ which do not dominate $P$. On the remaining
dominating components, quasi-finiteness gives finiteness over a nonempty open
subset, and generic flatness gives flatness after shrinking. Since all
geometric fibers are finite reduced groups in characteristic $0$, the
resulting finite flat morphism is finite \'{e}tale.

The inclusion $\mathcal{G}\subset\Aut(F)$ for all $F\in P$ gives a constant
subgroup family $\mathcal{G}\times U\subset\mathcal{A}_U$. Since on $U$ we
have $\mathcal{G}\subsetneq\Aut(F)$, the finite \'{e}tale group scheme
$\mathcal{A}_U\to U$ has degree strictly larger than $|\mathcal{G}|$. Since
$\mathcal{A}_U\to U$ is a finite \'{e}tale group scheme over the connected
base $U$, after passing to a connected finite \'{e}tale cover
$\tau:\widetilde U\to U$ we may trivialize it as a constant family of finite
groups:
\[
\mathcal{A}_U\times_U\widetilde U\simeq H\times\widetilde U,
\]
where $H$ is a finite abstract group. Under this trivialization, the subgroup
family $\mathcal{G}\times\widetilde U$ corresponds to a fixed subgroup
$H_{\mathcal{G}}\subset H$ which is identified with $\mathcal{G}$.

The natural inclusion $\mathcal{A}_U\subset\GL(n,\CC)\times U$ gives, after
pulling back to $\widetilde U$, a family of faithful representations
\[
\rho_{\widetilde u}:H\longrightarrow\GL(n,\CC),
\qquad \widetilde u\in\widetilde U.
\]
Since the isomorphism class of a representation of a finite group is discrete,
all the representations $\rho_{\widetilde u}$ are conjugate. Fix one point
$\widetilde u_0\in\widetilde U$, and define
$\mathcal{G}'=\rho_{\widetilde u_0}(H)$. Since
$\rho_{\widetilde u_0}(H_{\mathcal G})=\mathcal G$, we have
$\mathcal{G}\subsetneq\mathcal{G}'$, and naturally every $F\in U$ is linearly
equivalent to an element of $P_{n,d}^{\mathcal{G}',\mathrm{sm}}$.
Equivalently,
\[
U\subset P\cap\Sigma_{\mathcal{G}'},\qquad
\Sigma_{\mathcal{G}'}=
P_{n,d}^{\mathcal{G}',\mathrm{sm}}\cdot\GL(n,\CC)
\subset P_{n,d}^{\mathrm{sm}}.
\]
For the required closedness, put $X=P_{n,d}^{\mathrm{sm}}$,
$G=\GL(n,\CC)$, and $H=\mathcal{G}'$, and let $q:X\to X/\!/G$ be the
affine quotient. The discriminant complement $X$ is affine, and the natural
morphism $X^H/\!/N_G(H)\to X/\!/G$ is finite by
\cite[Theorem~1.1]{BateGeraniosMartin}. Its image is therefore closed.
Since smooth hypersurfaces are stable, all $G$-orbits in $X$ are closed, so
$\Sigma_{\mathcal{G}'}$ is the inverse image of this closed set under $q$.
Thus $P\cap\Sigma_{\mathcal{G}'}$ is closed in $P$. It contains the
nonempty open subset $U$, and hence equals the irreducible space $P$.
This proves $(1)\Rightarrow(3)$.
\end{proof}

Every triple $(G,\rho,\xi)$ introduced at the beginning of
Section~\ref{subsect:moduli-dim} determines the finite subgroup
$\mathcal G=\ker(\xi:\pi^{-1}(\rho(G))\to\CC^\times)$ of $\GL(6,\CC)$.
The space of cubic forms fixed by $\mathcal G$ is
$\mathcal V_{G,\rho,\xi}$, so Proposition~\ref{prop:GL_family} applies.

\section{Moduli and Lattice}\label{sect:latt}

Note that $\Aut^s(X)\lhd\Aut(X)$ for a smooth cubic fourfold $X$. For convenience we give the following definitions.

\begin{defn}
    Let $X$ be a smooth cubic fourfold. Its non-symplectic index $m(X)$ is defined to be the index $[\Aut(X):\Aut^s(X)]$.
\end{defn}

\begin{defn} 
    Let $G$ be a finite group, $m$ be a positive integer, $\rho:G\hookrightarrow \PGL(6,\CC)$ be an embedding, and $\xi:\pi^{-1}(\rho(G))\to\CC^\times$ be a character of the inverse image of $\rho(G)$ in $\GL(6,\CC)$.
    \begin{enumerate}
        \item We say that $(G,m,\rho,\xi)$ is admissible if there exists a smooth cubic fourfold $X$ of non-symplectic index $m$ such that $\Aut^s(X)=\rho(G)$ and $\pi^{-1}(\rho(G))$ acts on the defining equation of $X$ through the character $\xi$.
        \item We say $(G,m)$ is admissible, or $m$ is admissible for $G$, if there exist $\rho$ and $\xi$ such that $(G,m,\rho,\xi)$ is admissible.
    \end{enumerate}
\end{defn}

\subsection{Moduli Spaces and Dimension-Counting Formula}\label{subsect:moduli-dim}

Suppose a finite group $G$ acts by $\rho:G\hookrightarrow \PGL(6,\CC)$ on $\PP^5$ and preserves some smooth cubic fourfold $X\subset \PP^5$; we identify $G$ with $\rho(G)$. Suppose also that $\pi^{-1}(\rho(G))$ acts on the defining equation of $X$ by the scalar character $\xi:\pi^{-1}(\rho(G))\to\CC^\times$. Let $G^s=G\cap \Aut^s(X)\lhd G$. Then the induced action on $H^{3,1}(X,\CC)$ gives a character $\chi:G\to U(1)\subset\CC^\times$ which factors through the quotient map $G\to G/G^s$. Since $\Aut(X)/\Aut^s(X)$ is finite and can be embedded into $U(1)\subset \CC^\times$, it is isomorphic to the group of $m$-th roots of unity $\mu_m$, where $m=[\Aut(X):\Aut^s(X)]$. We have a short exact sequence
\[
    1\to \Aut^s(X)\to\Aut(X)\to\mu_m\to 1,
\]
and
\[
    1\to G^s\to G\to G/G^s\to 1,
\]
where $G/G^s$ is a (cyclic) subgroup of $\mu_m$.

Moreover, if we fix $(G,\rho,\xi)$, smooth cubic fourfolds preserved by $(G,\rho)$ through the scalar $\xi$ share the same character $\chi:G\to \CC^\times$ acting on $H^{3,1}$. From \cite[Proposition 4.1]{yu2020moduli}, one could construct the period domain $\DD_{(G,\rho,\xi)}$ from the $\chi$-eigenspace
\[
(\Lambda_{0,\CC})^\chi
=\{x\in \Lambda_{0,\CC}\mid gx=\chi(g)x\text{ for all }g\in G\},
\]
with a Hermitian form defined by the intersection form.

As a corollary, if $\chi$ is real (i.e. $\chi=\overline{\chi}$, or equivalently $\chi(G)\subset\{\pm 1\}$), then the corresponding Hermitian space has signature $(\dim (\Lambda_{0,\CC})^\chi-2,2)$ and $\DD_{(G,\rho,\xi)}$ is a type IV symmetric domain of dimension $\dim(\Lambda_{0,\CC})^\chi-2$; if $\chi$ is non-real (i.e. $\chi\neq\overline{\chi}$), then the corresponding Hermitian space has signature $(\dim(\Lambda_{0,\CC})^\chi-1,1)$ and $\DD_{(G,\rho,\xi)}$ is a complex hyperbolic ball of dimension $\dim (\Lambda_{0,\CC})^\chi-1$. 

On the other hand, from \cite[Theorem 1.1(i)]{yu2020moduli} we can define the moduli space $\mathcal{F}_{G,\rho,\xi}$ of smooth cubic fourfolds preserved by $G$ through the scalar action of $\pi^{-1}(\rho(G))$ on the defining equation with character $\xi$. We sketch the construction as follows:

Let $\mathcal{V}_{G,\rho,\xi}$ be the space of degree-$3$ polynomials in six variables on which $\pi^{-1}(\rho(G))$ acts through $\xi$, and let $\PP \mathcal{V}_{G,\rho,\xi}$ be its projectivization.

Let $\widehat{\rho(G)}$ be the preimage of $\rho(G)$ in $\SL(6,\CC)$ and $\lambda$ be the restriction of $\xi$ to $\widehat{\rho(G)}$. Define 
\[
N^{\lambda}_{\rho}(G)=\{g\in \SL(6,\CC)\ | \ g(\widehat{\rho(G)})g^{-1}=\widehat{\rho(G)}, \lambda(ghg^{-1})=\lambda(h), \forall h\in \widehat{\rho(G)}\},
\]
which is a reductive group and acts naturally on $\PP \mathcal{V}_{G,\rho,\xi}$.

Denote by $\mathcal{V}_{G,\rho,\xi}^{\mathrm{sm}}$ and $\mathcal{V}_{G,\rho,\xi}^{\mathrm{ss}}$ the smooth and semistable elements in $\mathcal{V}_{G,\rho,\xi}$ under the action of $N^{\lambda}_{\rho}(G)$, and by $\PP\mathcal{V}_{G,\rho,\xi}^{\mathrm{sm}}$ and $\PP\mathcal{V}_{G,\rho,\xi}^{\mathrm{ss}}$ their projectivizations. Define $\mathcal{F}_{G,\rho,\xi}$ to be the GIT quotient $N^{\lambda}_{\rho}(G)\dbs \PP\mathcal{V}_{G,\rho,\xi}^{\mathrm{sm}}$, with GIT compactification $\overline{\mathcal{F}}_{G,\rho,\xi}=N^{\lambda}_{\rho}(G)\dbs \PP\mathcal{V}_{G,\rho,\xi}^{\mathrm{ss}}$.

\begin{prop}[{\cite[Theorem 1.1(i)]{yu2020moduli}}]
\label{proposition: moduli_dimension}
 The dimension of $\mathcal{F}_{G,\rho,\xi}$ equals the dimension of $\DD_{(G,\rho,\xi)}$.
\end{prop} 

The cubic fourfolds in this moduli space have the same $G$-equivariant lattice data.

\subsection{Discriminant Forms} Following the notation in Section~\ref{subsect:ST}, let additionally $G=\Aut^s(X)$. Recall that there are two embeddings $\Lambda_0\oplus E_6\hookrightarrow\mathrm{II}_{26,2}$ and $S\oplus T\hookrightarrow\Lambda_0$ with direct summands primitive. Thus we have the following embedding
\begin{equation}\label{eq:E6}
    S\oplus T\oplus E_6\hookrightarrow\mathrm{II}_{26,2},
\end{equation}
with three direct summands primitive. The following result is well-known. We give a proof for the reader's convenience.

\begin{prop}\label{prop:either_ST_oplus_E6_primitive}
    Suppose that $\mathrm{II}_{26,2}$ is an overlattice of $S\oplus T\oplus E_6$ in which each direct summand is primitive. Then one of the following two cases occurs:
    \begin{enumerate}
        \item $T\oplus E_6$ is primitive in $\mathrm{II}_{26,2}$. In this case $q_S=3^{-1}\oplus (-q_T)$.
        \item $S\oplus E_6$ is primitive in $\mathrm{II}_{26,2}$. In this case $q_T=3^{-1}\oplus (-q_S)$.
    \end{enumerate}
\end{prop}
\begin{proof}
Let $K=E_6^\perp\subset\mathrm{II}_{26,2}$. From $q_{E_6}\cong 3^{+1}$ and \cite[Proposition 1.5.1]{Nikulin_1980}, we know $q_K\cong 3^{-1}$. Again from \cite[Proposition 1.5.1]{Nikulin_1980}, there is an isotropic subgroup $H$ of $A_S\oplus A_T$ whose projections identify it with anti-isometric subgroups of $A_S$ and $A_T$, such that $q_K\cong(q_S\oplus q_T)|_{H^\perp}/H$. Thus $|H^\perp|=3|H|$ and $|A_S\oplus A_T|=3|H|^2$. Since $A_S$ and $A_T$ have anti-isometric subgroups isomorphic to $H$, one of the cases stated in the proposition must occur.
\end{proof}

Conversely, for each $G$ in the classification given by
\cite[Theorem~1.2]{laza2022automorphisms}, the signatures of $S$ and $T$ and
the form $q_S$ are determined. A fixed pair $(G,S)$ can admit more than one
possibility for $T$ and $q_T$. As a consequence of
\cite[Theorem~4.5]{laza2022automorphisms}, all such possibilities can be
described in lattice-theoretic terms:

\begin{lem}\label{lem:T_realize}
    For a fixed pair $(G,S)$, a lattice $T$ of signature $(20-\rank(S),2)$ can be realized if and only if there is an embedding $S\oplus T\oplus E_6\hookrightarrow\mathrm{II}_{26,2}$.
\end{lem}

Therefore, using Proposition~\ref{prop:either_ST_oplus_E6_primitive}, Lemma~\ref{lem:T_realize}, and \cite[Theorem 1.10.1]{Nikulin_1980}, we can explicitly compute all possibilities of $(q_S,q_T)$.

\subsection{Leech Subpairs}
We have some other propositions comparing two saturated Leech pairs $(G_{1},S_{1})$ and $(G_{2},S_{2})$ that appear in the classification of symplectic automorphism groups of cubic fourfolds \cite{laza2022automorphisms}.

By $(G_{1},S_{1})\leq (G_{2},S_{2})$, we mean that $G_{1}<G_{2}$ and $S_{1}\subset S_{2}$. Now $S_{1}$ and $S_{2}$ are coinvariant lattices in $\Lambda\cong H^{4}(X,\ZZ)$, with orthogonal complements $T_{1}$ and $T_{2}$ being invariant lattices.

\begin{lem}
\label{Lemma:normalize_latticecriterion}
    For $g\in O(\Lambda_{0})$ stabilizing $S_{2}$, the following are equivalent:
    \begin{itemize}
        \item $g$ normalizes $G_{1}$;
        \item $g(T_{1})=T_{1}$;
        \item $g(S_{1})=S_{1}$.
    \end{itemize}
\end{lem} 

\begin{proof}
    The first condition implies the second, while the second and third are equivalent by orthogonality.
    Take $h\in G_{1}$ and suppose that $g(S_{1})=S_{1}$. Then $g$ induces an action on $A_{S_{1}}$, and $ghg^{-1}$ acts trivially on $A_{S_{1}}$. The saturation of $(G_1,S_1)$ therefore gives $ghg^{-1}\in G_{1}$.
\end{proof}

\begin{cor}
    If $\rank(S_{2})=\rank(S_{1})+1$, then $N_{G_{2}}(G_{1})/G_{1}$ is either trivial or isomorphic to $C_{2}$.
\end{cor}

\begin{proof}
    Let $g\in G_{2}$ satisfy $gG_{1}g^{-1}=G_{1}$. Then $g$ acts on $S_{1}$ and $T_{1}$. Then $g$ has a well-defined action on $S_{1}^{\perp}$, the orthogonal complement of $S_{1}$ in $S_{2}$. Now $S_{1}^{\perp}$ is isomorphic to $\ZZ$ as an abelian group. Therefore, $g$ acts on $S_{1}^{\perp}$ by $\pm 1$. If $g$ acts on $S_{1}^{\perp}$ trivially, then $g$ acts trivially on $T_{1}$. Then $g\in G_{1}$ by the definition of saturation. So there is an exact sequence

    \begin{equation*}
        1\longrightarrow G_{1}\longrightarrow N_{G_{2}}(G_{1})\longrightarrow C_{2}.
    \end{equation*}
    and the corollary follows.
\end{proof}

\section{Restrictions and Existence of Indices}\label{sect:indice}

Let us use results established in \S~\ref{sect:latt} to give some general restrictions and existence of indices. A first strong restriction is

\begin{prop}\label{prop:2k3l}
    If $m$ is admissible for some $G$, then $m=2^k\cdot 3^l$, where $k\geq 0$ and $l\in\{0,1\}$.
\end{prop}
\begin{proof}

    Choose a realization $X$ of $(G,m)$, and let $\chi:\Aut(X)\to\mu_m$ be the character on $H^{3,1}(X,\CC)$. Choose $g\in\Aut(X)$ such that $\chi(g)=\zeta_m$ generates $\mu_m$. Then $m\mid\ord(g)$. By \cite[Proposition 3.5]{Zheng2020OnAA}, $\ord(g)$ divides one of $21$, $30$, $32$, $33$, $36$, or $48$. If a prime $p\neq2,3$ divides $m$, then $p^2\nmid\ord(g)$. Consequently, $g^{\ord(g)/p}$ has order $p$ and
    $\chi(g^{\ord(g)/p})=\zeta_m^{\ord(g)/p}\neq1$, so it is non-symplectic. This contradicts \cite[Proposition 6.3]{laza2022automorphisms}, according to which a non-symplectic automorphism of prime order has order $2$ or $3$. Therefore, $m=2^k3^l$ for some $k,l\geq0$.

    Suppose that $l\geq 2$. Since $\ord(g)$ divides one of $21$, $30$, $32$, $33$, $36$, or $48$, we must have $l=2$. Let $h\coloneqq g^{\ord(g)/9}$, which has order $9$. Since
    \[
    \chi(h^3)=\chi(g)^{\ord(g)/3}=\omega^{\ord(g)/m}\neq 1,
    \]
    the automorphism $h^3$ is non-symplectic; the last inequality follows because $3\nmid\ord(g)/m$. Mayanskiy's classification \cite[Corollary~1 and Table~2]{mayanskiy2013abelianautomorphismgroupscubic} shows that the cube of every order-$9$ automorphism of a smooth cubic fourfold has signature $(0,0,0,1,1,1)$, up to a linear change of coordinates, a permutation of the coordinates, and replacement of the generator by its inverse, and hence is the symplectic type $\mathcal{F}_3^4$ in \cite[Theorem~2.8]{gonzalez2011automorphisms}, a contradiction.
\end{proof}

\subsection{Reducing Indices}
\begin{lem}\label{lem:dim_reduction}
    Suppose $(G,m,\rho_0,\xi)$ is admissible and is realized by $X$. Identify $G$ with $\rho_0(G)=\Aut^s(X)$. Put $\widetilde G=\Aut(X)$, let $\rho:\widetilde G\hookrightarrow\PGL(6,\CC)$ be its natural action, so that $\rho|_G=\rho_0$, and let
    \[
    \widetilde\xi:\pi^{-1}(\rho(\widetilde G))\longrightarrow\CC^\times
    \]
    be the character on a defining equation of $X$. Its restriction to $\pi^{-1}(\rho_0(G))$ is $\xi$. Let $\chi:\widetilde G\twoheadrightarrow\mu_m$ be the character on $H^{3,1}(X,\CC)$, so that
    \[
    1\longrightarrow G\longrightarrow\widetilde G
    \xrightarrow{\chi}\mu_m\longrightarrow 1.
    \]
    Suppose that a prime $p$ divides $m$ and $m/p\neq 1,3$. Put
    \[
    \widetilde G'=\chi^{-1}(\mu_{m/p}),\qquad
    \rho'=\rho|_{\widetilde G'},\qquad
    \widetilde\xi'=\widetilde\xi|_{\pi^{-1}(\rho'(\widetilde G'))},\qquad
    \chi'=\chi|_{\widetilde G'}.
    \]
    Then $\chi':\widetilde G'\twoheadrightarrow\mu_{m/p}$ and
    \[
        \dim \mathcal{F}_{\widetilde G',\rho',\widetilde\xi'}\geq
        \begin{cases}2\dim \mathcal{F}_{\widetilde{G},\rho,\widetilde\xi}, & \text{if }\frac{m}{p}=2,\\
        2\dim \mathcal{F}_{\widetilde G,\rho,\widetilde\xi}+1, &\text{otherwise}
        \end{cases}.
    \]
\end{lem}
\begin{proof}
    The key is to show that
    \begin{equation}\label{eq:2dim}
        \dim (\Lambda_{0,\CC})^{\chi^{\prime}}\geq 2\dim(\Lambda_{0,\CC})^{\chi}.
    \end{equation}
    By Proposition~\ref{prop:2k3l}, $p\in\{2,3\}$. Let $\varphi$ denote
    the induced $\mu_m$-action on $T\otimes\CC$. For each
    $0\leq r\leq p-1$, its $\zeta_p^r\zeta_m$-eigenspace is contained in
    $(\Lambda_{0,\CC})^{\chi'}$, since
    $(\zeta_p^r\zeta_m)^p=\zeta_m^p$. If $p=2$, take $r_0=1$. If $p=3$,
    write $m=3\cdot2^k$ and choose $r_0\in\{1,2\}$ such that
    $3\nmid1+r_0 2^k$. Then $\zeta_p^{r_0}\zeta_m$ is a primitive
    $m$-th root of unity. Since the representation comes from the lattice
    $\Lambda_0$, Galois conjugation shows that its
    $\zeta_p^{r_0}\zeta_m$- and $\zeta_m$-eigenspaces have the same
    dimension, namely $\dim(\Lambda_{0,\CC})^\chi$. They have zero
    intersection, which proves \eqref{eq:2dim}.

    Since $m>2$, the character $\chi$ is non-real, whereas $\chi^{\prime}$ is real if and only if $\frac{m}{p}=2$. Proposition~\ref{proposition: moduli_dimension} therefore gives
    \[
    \dim \mathcal{F}_{\widetilde G,\rho,\widetilde\xi}=\dim (\Lambda_{0,\CC})^{\chi}-1,
    \]
    and
    \[
    \dim \mathcal{F}_{\widetilde G',\rho',\widetilde\xi'}=
        \begin{cases}
            \dim (\Lambda_{0,\CC})^{\chi^{\prime}}-2, &\text{if }\frac{m}{p}=2,\\
            \dim (\Lambda_{0,\CC})^{\chi^{\prime}}-1, &\text{otherwise}
        \end{cases}.
    \]
    Combining with \eqref{eq:2dim} we conclude the result.
\end{proof}

On the other hand, strict inclusion of moduli would give a reduction on indices in the following way.

\begin{prop}\label{prop:ind_red}
    Suppose $(G,m,\rho_0,\xi)$ is admissible and is realized by $X$. Identify $G$ with $\rho_0(G)=\Aut^s(X)$. Put $\widetilde G=\Aut(X)$, let $\rho:\widetilde G\hookrightarrow\PGL(6,\CC)$ be its natural action, so that $\rho|_G=\rho_0$, and let
    \[
    \widetilde\xi:\pi^{-1}(\rho(\widetilde G))\longrightarrow\CC^\times
    \]
    be the character on a defining equation of $X$. Its restriction to $\pi^{-1}(\rho_0(G))$ is $\xi$. Let $\chi:\widetilde G\twoheadrightarrow\mu_m$ be the character on $H^{3,1}(X,\CC)$. Suppose that a prime $p$ divides $m$, and put
    \[
    \widetilde G'=\chi^{-1}(\mu_{m/p}),\qquad
    \rho'=\rho|_{\widetilde G'},\qquad
    \widetilde\xi'=\widetilde\xi|_{\pi^{-1}(\rho'(\widetilde G'))}.
    \]
    If
    \[
    \dim\mathcal{F}_{\widetilde G,\rho,\widetilde\xi}
    <\dim\mathcal{F}_{\widetilde G',\rho',\widetilde\xi'},
    \]
    then $(G,m/p,\rho_0,\xi)$ is admissible.
\end{prop} 
\begin{proof}
The linear groups determined by
$(\widetilde G,\rho,\widetilde\xi)$ and
$(\widetilde G',\rho',\widetilde\xi')$ have index $p$. If the latter is
not the linear automorphism group of a general member of its family,
Proposition~\ref{prop:GL_family} gives a fixed proper overgroup preserving
every member up to linear equivalence. Applied to $X$, this overgroup is,
after conjugation, contained in the former linear group. Since the index is
prime, it must equal that group. The two families would then have the same
image in the moduli space, contrary to the strict dimension inequality.
Thus a general member has projective automorphism group $\widetilde G'$.
Its character on $H^{3,1}$ is $\chi|_{\widetilde G'}$, whose kernel is
$G$. Hence $(G,m/p,\rho_0,\xi)$ is admissible.
\end{proof}

Combining Lemma~\ref{lem:dim_reduction} and Proposition~\ref{prop:ind_red}, we deduce that

\begin{cor}\label{cor:reduce}
    Let $G$ be a group.
    \begin{enumerate}
        \item Let $k\geq 3$ be an integer. If $2^k$ is admissible for $G$, so are $2^{k-1},2^{k-2},\ldots,2$.
        \item Let $k\geq 2$ be an integer. If $2^k\cdot 3$ is admissible for $G$, so are $2^{k-1}\cdot 3,\ldots,6$ and $2^k,\ldots,2$.
        \item If $m=4$ or $6$ is admissible for $G$ and there is a realization whose Hodge character $\chi:\Aut(X)\to\mu_m$ satisfies $\dim(\Lambda_{0,\CC})^\chi>1$, then $2$ is admissible for $G$.
    \end{enumerate}
\end{cor}
\begin{proof}
    Apply Proposition~\ref{prop:ind_red} successively, with $p=2$ or $3$
    as appropriate. Lemma~\ref{lem:dim_reduction} gives the required strict
    inequality except that the reduction from $4$ or $6$ to $2$ requires
    the preceding family to have positive dimension. In (1) and (2), this
    follows from the previous reduction; in (3), it follows from the stated
    condition on $(\Lambda_{0,\CC})^\chi$.
\end{proof}

The following is a refined version of Corollary~\ref{cor:reduce}, and the proof is the same.

\begin{cor}\label{cor:reduce_refined}
    Suppose $(G,m,\rho_0,\xi)$ is admissible and is realized by $X$.
    Let $\chi$ be the character of $\Aut(X)$ on $H^{3,1}(X,\CC)$; its
    image is $\mu_m$.
    Assume that one of the following holds:
    \[
        m^{\prime}\in\begin{cases}
        \{2^{a-1},2^{a-2},\ldots,2\}, & \text{if }m=2^a,\ a\geq3,\\
        \{3\cdot2^{a-1},\ldots,6\}\cup \{2^a,2^{a-1},\ldots,2\}, & \text{if }m=3\cdot2^a,\ a\geq2,\\
        \{2\}, & \text{if }m\in\{4,6\}\text{ and }\dim(\Lambda_{0,\CC})^\chi>1.
    \end{cases}.
    \]
    Then $(G,m',\rho_0,\xi)$ is admissible.
\end{cor}

\subsection{Generic Indices}

For a lattice $T$ of signature $(d,2)$, let $\DD_T$ denote one connected
component of its type IV period domain.

\begin{lem}\label{lem:pm1}
    Let $T$ be a lattice of signature $(d,2)$ with $d\geq1$, let
    $\iota\in O(T)$ have finite order, and let $\lambda\neq 1$ be an
    eigenvalue of $\iota$. If
    \[
        \DD_T(\iota,\lambda)
        =\{[\omega]\in\DD_T\mid \iota(\omega)=\lambda\omega\}
    \]
    is nonempty and has dimension $d$, then $\lambda=-1$ and
    $\iota=-\id$.
\end{lem}

\begin{proof}
    Let $V_\lambda=\ker(\iota_{\CC}-\lambda\id)$. If $\lambda$ is real,
    then $\lambda=-1$, and the corresponding type IV domain has dimension
    $\dim V_\lambda-2$. The dimension assumption and
    $d=\rank(T)-2$ give
    $V_\lambda=T_{\CC}$, so $\iota=-\id$. If $\lambda$ is non-real, the
    corresponding complex ball has dimension $\dim V_\lambda-1$.
    The dimension assumption gives $\dim V_\lambda=\rank(T)-1$. The
    disjoint conjugate eigenspace $V_{\overline\lambda}$ has the same dimension, so
    $2(\rank(T)-1)\leq\rank(T)$, contrary to $\rank(T)\geq3$.
\end{proof}

In what follows, $(G,\rho,\xi)$ is a triple such that $\rho(G)=\Aut^s(X)$ for some smooth cubic fourfold $X$ and $\pi^{-1}(\rho(G))$ acts on a defining polynomial through $\xi$; we identify $G$ with $\rho(G)$. Let $\mathcal{F}_{G,\rho,\xi}$ denote the moduli space of smooth cubic fourfolds admitting this action. When $\rank(S)<20$, such triples correspond bijectively to the individual rows of Table~\ref{table: main}. This condition gives a positive-dimensional family and is equivalent to $\rank(T)\geq3$.

If $T\oplus E_6$ is primitive in $\mathrm{II}_{26,2}$ under the fixed
embedding \eqref{eq:E6}, we fix the gluing anti-isometry
$A_S\longrightarrow A_{E_6}\oplus A_T$ and use it to regard
$\id_{A_{E_6}}\oplus(-\id_{A_T})$ as an involution of $q_S$.

\begin{prop-def}\label{prop-def:generic1}
    For a triple $(G,\rho,\xi)$ with $\rank(S)<20$, the following are equivalent:
    \begin{enumerate}
        \item Generic cubic fourfolds in $\mathcal{F}_{G,\rho,\xi}$ have non-symplectic index $1$.
        \item There does not exist an automorphism of $\Lambda$ acting trivially on $\langle h^2\rangle$ and by $-\id$ on $T$.
        \item Either there exists some nontrivial isotropic subgroup $H\subset 3^{+1}\oplus A_T$ such that $q_S\cong -((3^{+1}\oplus q_T)|_{H^\perp})/H$, or $T\oplus E_6$ is primitive in $\mathrm{II}_{26,2}$ and the involution of $q_S$ specified above does not lie in the image of the canonical homomorphism $O(S)\to O(A_S)$.
    \end{enumerate}
    In this case, we say $(G,\rho,\xi)$ has generic non-symplectic index $1$.
\end{prop-def}
\begin{proof}
    $(1)\Rightarrow(2)$: We prove the contrapositive. Assume that $\phi\in O(\Lambda)$ acts trivially on $\langle h^2\rangle$ and by $-\id$ on $T$. For a generic $X^{\prime}$, we have $\Aut^s(X')=\rho(G)$. Since $\phi$ preserves $H^{3,1}(X^{\prime},\CC)$ and $H^{1,3}(X^{\prime},\CC)$, both of which are contained in $T$, the global Torelli theorem implies that $\phi$ is induced by an automorphism $\widetilde{\phi}\in\Aut(X^{\prime})$. This automorphism is non-symplectic, whereas $\widetilde{\phi}^2$ is symplectic. Thus the generic non-symplectic index is not $1$.

    $(2)\Rightarrow(1)$: Suppose that the generic non-symplectic index is
    greater than $1$. Proposition~\ref{prop:GL_family} gives a fixed finite
    linear overgroup whose family has the same image in moduli as
    $\mathcal F_{G,\rho,\xi}$. Since the generic symplectic automorphism
    group is $G$, this overgroup contains a non-symplectic element $f$.
    Let $\iota=f^*|_T$, and let $\lambda\neq1$ be its eigenvalue on the
    period line. Proposition~\ref{proposition: moduli_dimension} gives
    $\dim\DD_T(\iota,\lambda)=\dim\DD_T$. Lemma~\ref{lem:pm1} gives
    $\iota=-\id$, contradicting (2).
    
    $(2)\Leftrightarrow(3)$: Since (2) is complementary to Proposition-Definition~\ref{prop-def:generic2}(2), it remains to show that (3) is complementary to Proposition-Definition~\ref{prop-def:generic2}(3). In \eqref{eq:E6}, let $K=(S)^\perp_{\mathrm{II}_{26,2}}$ be the saturation of $T\oplus E_6$ in $\mathrm{II}_{26,2}$. By \cite[Proposition 1.4.1]{Nikulin_1980}, there is an isotropic subgroup $H$ of $A_{E_6}\oplus A_T$ such that $A_K\cong H^\perp/H$ and $q_S\cong -q_K\cong -((3^{+1}\oplus q_T)|_{H^\perp})/H$, since $q_{E_6}\cong 3^{+1}$. Moreover, $H$ is trivial if and only if $K=T\oplus E_6$, or equivalently $q_S\cong -(3^{+1}\oplus q_T)$.
\end{proof}

\begin{prop-def}\label{prop-def:generic2}
    For a triple $(G,\rho,\xi)$ with $\rank(S)<20$, the following are equivalent:
    \begin{enumerate}
        \item Generic cubic fourfolds in $\mathcal{F}_{G,\rho,\xi}$ have non-symplectic index 2.
        \item There exists an automorphism of $\Lambda$ acting trivially on $\langle h^2\rangle$ and by $-\id$ on $T$.
        \item The sublattice $T\oplus E_6$ is primitive in $\mathrm{II}_{26,2}$, and the involution of $q_S$ specified above lies in the image of the canonical homomorphism $O(S)\to O(A_S)$.
    \end{enumerate}
    In this case we say $(G,\rho,\xi)$ has generic non-symplectic index $2$.
\end{prop-def}
\begin{proof}
    $(1)\Rightarrow (2)$ is contained in ``$(2)\Rightarrow(1)$'' in Proposition-Definition \ref{prop-def:generic1}.
    
    $(2)\Rightarrow(1)$: The proof of ``$(1)\Rightarrow(2)$'' in
    Proposition-Definition~\ref{prop-def:generic1} shows that the generic
    non-symplectic index is divisible by $2$. Consider the fixed generic
    automorphism group obtained in the proof of
    Proposition~\ref{prop:GL_family}; its family has the same image in
    moduli. For every non-symplectic element $f$,
    Proposition~\ref{proposition: moduli_dimension} and
    Lemma~\ref{lem:pm1} give $f^*|_T=-\id$. Thus every nontrivial value of
    the character on the period line is $-1$. Since the quotient by the
    symplectic subgroup embeds into $\CC^\times$, it has order at most $2$,
    and hence exactly $2$.

    $(2)\Rightarrow (3)$: Let $\phi$ be as in (2). Every element of
    $(S)^\perp_{\Lambda}$ can be written as $ch^2+x$, where $c\in \QQ$ and
    $x\in T_\QQ$. Since
    \[
    (S)^\perp_{\Lambda}
    \subset (\langle h^2\rangle\oplus T)^*
    \subset \frac{1}{3}\langle h^2\rangle\oplus T_\QQ,
    \]
    we have $c\in\frac{1}{3}\ZZ$. Since
    $(ch^2+x)+\phi(ch^2+x)=2ch^2\in\Lambda$, we also have
    $c\in\frac{1}{2}\ZZ$. Thus $c\in\ZZ$ and $x\in T$, so
    $\langle h^2\rangle\oplus T=(S)^\perp_\Lambda$. The standard gluings
    along $A_{\Lambda_0}$ identify
    \[
    (S^\perp_\Lambda)/(\langle h^2\rangle\oplus T)
    \cong
    (S^\perp_{\mathrm{II}_{26,2}})/(E_6\oplus T).
    \]
    Thus $T\oplus E_6=S^\perp_{\mathrm{II}_{26,2}}$, so it is primitive.
    Since $\phi$ fixes $h^2$, the $\langle h^2\rangle$--$\Lambda_0$
    gluing shows that $\phi|_{\Lambda_0}$ acts trivially on
    $A_{\Lambda_0}$. Hence
    \cite[Corollary 1.5.2]{Nikulin_1980}
    extends $\id_{E_6}\oplus\phi|_{\Lambda_0}$ to
    $\mathrm{II}_{26,2}$. Under the chosen gluing anti-isometry,
    $\phi|_S$ induces the involution of $q_S$ specified above. This proves
    (3).
    
    $(3)\Rightarrow (2)$: Choose an isometry of $S$ inducing the specified
    involution of $q_S$. By \cite[Corollary 1.5.2]{Nikulin_1980}, it extends
    together with $\id_{E_6}\oplus(-\id_T)$ to an isometry $\varphi$ of
    $\mathrm{II}_{26,2}$. Since $\varphi|_{E_6}=\id_{E_6}$, its restriction
    to $\Lambda_0$ acts trivially on $A_{\Lambda_0}$. Hence
    $\id_{\langle h^2\rangle}\oplus\varphi|_{\Lambda_0}$ preserves the
    index-three gluing subgroup defining $\Lambda$ and extends to an
    isometry of $\Lambda$, proving (2).
\end{proof}

\begin{cor}\label{cor:gen12}

    Let $(G,\rho,\xi)$ be a triple as above, with $\rank(S)<20$.
    \begin{enumerate}
        \item The generic non-symplectic index of $\mathcal{F}_{G,\rho,\xi}$ is either $1$ or $2$. When it is $2$, every smooth cubic fourfold in the moduli space admits a non-symplectic automorphism whose square is symplectic.
        \item More explicitly, the generic index is $2$ if and only if $q_S\cong -(3^{+1}\oplus q_T)$, and in this case any smooth cubic fourfold in the moduli space admits a non-symplectic involution.
    \end{enumerate}
\end{cor}

\begin{proof}
    The first part is direct from Proposition-Definitions~\ref{prop-def:generic1} and \ref{prop-def:generic2} and Proposition~\ref{prop:either_ST_oplus_E6_primitive}.
    
    For the second part, Lemma~\ref{lem:T_realize} and
    \cite[Theorem~1.10.1]{Nikulin_1980} determine the finite possibilities
    for $q_T$ displayed in Table~\ref{table: main}. Koike's classification
    \cite{KOIKE202512,KOIKE2026} determines the corresponding lattice $T$
    uniquely for each irreducible family. For fixed $S$ with $\rank(S)<20$, the
    number of these families equals the number of possible forms $q_T$,
    except for $G=S_{3,3}$, where one form $q_T$ corresponds to two
    families.

    The components satisfying $q_S\cong -(3^{+1}\oplus q_T)$ are given in
    \cite[(2.6), (3.3), (4.2), (4.6), (4.14), (5.2), (6.1),
    (6.5), (6.9), (6.11), (6.13), and Sections~7.1, 7.3, and 7.5]{KOIKE202512}
    and \cite[(0.1) and (0.2)]{KOIKE2026}. The equations and actions given
    there show that every member of each component admits a non-symplectic
    involution; this holds for both components corresponding to $S_{3,3}$.
    Proposition-Definitions~\ref{prop-def:generic1} and
    \ref{prop-def:generic2} now give the result.
\end{proof}

\begin{rmk}
We take this opportunity to correct an error regarding the case $G=S_3$ presented in \cite{laza2022automorphisms}, which is also pointed out in \cite{KOIKE2026}. Specifically, there is a computational error in \cite[Lemma 4.24]{laza2022automorphisms}: in the second case, the dimension of the space of invariant polynomials should be $17$, rather than $15$; consequently, the family has dimension $6$, rather than $4$. Since $6=20-\rank(S)$ and the family contains the Fermat cubic fourfold, its smooth locus is nonempty and has generic symplectic automorphism group $S_3$. Moreover, the entire family admits the non-symplectic involution $(12)$. Its defining polynomials lie in
\[
\Span\{s_1^3,s_1s_2,s_3, s_1^2 x_5,s_1^2 x_6,s_2 x_5,s_2 x_6,s_1 x_4^2,s_1 x_5^2,s_1 x_5 x_6,s_1x_6^2,x_4^2 x_5,x_4^2 x_6,x_5^3,x_5^2 x_6,x_5x_6^2,x_6^3\},
\] 
where $s_k=\sum_{i=1}^3 x_i^k$. Consequently, there are in fact two distinct families of smooth cubic fourfolds with generic symplectic automorphism group $S_3$. This computational error propagates to subsequent results: \cite[Theorem 1.2(3)(b)]{laza2022automorphisms} and \cite[Section 3.1(v)]{KOIKE202512} incorrectly state that there is only one such family.
\end{rmk}

\section{Direct Calculations}\label{sect:calc}

\subsection{GAP-Assisted Computation}

Let $X$ be a smooth cubic fourfold. Then $\Aut^s(X)$ must be one of the $34$ groups listed in \cite[Theorem 1.2]{laza2022automorphisms}, and $\Aut(X)$ must be a subgroup of one of the maximal groups listed in \cite[Theorem 1.2]{yang2024automorphism}. Moreover, Proposition~\ref{prop:2k3l} implies that the non-symplectic index has the form $2^k3^l$, where $l\in\{0,1\}$. Using GAP \cite{GAP4}, the calculation gives the following necessary bounds.
\begin{prop}\label{prop:YYZ_bounds}
    Let $X$ be a smooth cubic fourfold. For a fixed isomorphism class of $\Aut^s(X)$, the non-symplectic index of $X$ must divide one of the group-theoretic bounds in Table~\ref{table: main}. Moreover, $\Aut(X)$ must be isomorphic to one of the candidate groups listed for $\Aut^s(X)$ in \path{gap_yyz_bounds/gap_yyz_bounds.md}\cite{FuWangZhengAutcub4fold2026}.
\end{prop}

Algorithm~\ref{GAP} summarizes the computation. See also the ancillary file \path{anc/gap_yyz_bounds.txt}, and see \path{gap_yyz_bounds}\cite{FuWangZhengAutcub4fold2026} for more details.

\begin{algorithm}
\caption{Computing the group-theoretic candidates}\label{GAP}
\begin{algorithmic}
\Require $\texttt{Symplectic\allowbreak Groups, Maximal\allowbreak Groups}$
\State $\texttt{Candidates} \gets \varnothing$
\For{$M\in\texttt{Maximal\allowbreak Groups}$}
\For{$[H]\in\texttt{ConjugacyClasses\allowbreak Subgroups}(M)$}
\For{$G\in\texttt{Symplectic\allowbreak Groups}$}
    \If{$|G|\mid |H|$}
        \State $m\gets |H|/|G|$
        \If{$m=2^k$ or $m=3\cdot2^k$ for some $k\geq0$, and there is $N\lhd H$ with $N\cong G$ and $H/N$ cyclic}
            \State add $(G,m,H,M)$ to $\texttt{Candidates}$
        \EndIf
    \EndIf
\EndFor
\EndFor
\EndFor
\State retain one representative for each recorded abstract-group datum
\State \Return $\texttt{Candidates}$
\end{algorithmic}
\end{algorithm}

\begin{rmk}
    The index alone need not determine the abstract full automorphism group. The complete output records the candidate groups and the maximal groups in which they occur.
\end{rmk}

\subsection{GIT Calculation}
\label{subsection:GITcalculation}

Let $G$ be a subgroup of $\PGL(6,\CC)$ and let $\widehat{G}$ be its preimage in $\SL(6,\CC)$ with character $\lambda$. 

We have the following exact sequence:

\begin{equation*}
    1\longrightarrow C_{\PGL(6,\CC)}(G) \longrightarrow N_{\PGL(6,\CC)}(G) \longrightarrow \Aut(G).
\end{equation*}

However, this exact sequence is not always suitable for our setting, because $N_{\PGL(6,\CC)}(G)$ does not have a well-defined action on the family of cubic fourfolds with $G$-action.

We should restrict to a subgroup of $N_{\PGL(6,\CC)}(G)$ as follows, see \cite[Section~2.2]{yu2020moduli}.

Let 
\[
N^{\lambda}_{\SL(6,\CC)}(G)= \{g\in \SL(6,\CC)\mid g\widehat{G}g^{-1}=\widehat{G}\ , \ \lambda(gg_{1}g^{-1})=\lambda(g_{1}),\ \forall g_{1}\in \widehat{G} \},
\]
and 
\[
C^{\lambda}_{\SL(6,\CC)}(G)= \{g\in \SL(6,\CC)\mid gg_{1}g^{-1}=\mu(g_1) g_{1},\ \mu(g_1) \in \mu_{6}, \ \lambda(gg_{1}g^{-1})=\lambda(g_{1})\ , \forall g_{1}\in \widehat{G} \}.
\]
Let $N^{\lambda}(G)$ and $C^{\lambda}(G)$ be the images of $N^{\lambda}_{\SL(6,\CC)}(G)$ and $C^{\lambda}_{\SL(6,\CC)}(G)$ in $\PGL(6,\CC)$, respectively. We omit $\lambda$ if there is no ambiguity in the character.

Similarly, we have the following exact sequence:
\begin{equation}
\label{equation:computeNlambda}
    1\longrightarrow C^{\lambda}(G) \longrightarrow N^{\lambda}(G) \longrightarrow \Aut(G).
\end{equation}

Let $X$ be a smooth cubic fourfold with $\Aut^s(X)=G<\PGL(6,\CC)$, where $\widehat{G}$ acts on $X$ via the character $\lambda$. Let $\PP\mathcal{V}_{G}^{\lambda}$ be the family of equations of cubic fourfolds with symplectic $G$-action. Then $N^{\lambda}_{\SL(6,\CC)}(G)$, and hence $N^{\lambda}(G)$, acts naturally on $\PP\mathcal{V}^{\lambda}_{G}$ by \cite[Lemma 2.5]{yu2020moduli}. Let $\mathrm{Stab}_{N^{\lambda}(G)}(X)$ be the stabilizer of $X\in\PP\mathcal{V}^{\lambda}_{G}$ under this action.

We have the following lemma.
\begin{lem}
\label{lem:stabAut}
For a smooth cubic fourfold $X\in \PP\mathcal{V}_{G}^{\lambda}$, the stabilizer $\mathrm{Stab}_{N^{\lambda}(G)}(X)$ is described as follows:

\begin{enumerate}
    \item If $\Aut^{s}(X)=G$, $\mathrm{Stab}_{N^{\lambda}(G)}(X)=\Aut(X)$. 
    \item If $G\subsetneq \Aut^{s}(X)$, $\mathrm{Stab}_{N^{\lambda}(G)}(X)=N_{\Aut(X)}(G)$.
\end{enumerate}
\end{lem}
\begin{proof}
    If $\Aut^{s}(X)=G$, then $\mathrm{Stab}_{N^{\lambda}(G)}(X)=\Aut(X)$. This follows from \cite[Lemma 4.1]{he2025cubicfourfoldsorder7automorphism}.

    For the other case, since $\lambda$ extends to $\widehat{\Aut}(X)$, the preimage of $\Aut(X)$ in $\SL(6,\CC)$, we know that $N_{\Aut(X)}(G)\subset \mathrm{Stab}_{N^{\lambda}(G)}(X)$. On the other hand, every element of $\mathrm{Stab}_{N^{\lambda}(G)}(X)$ is an automorphism of $X$ normalizing $G$. So $\mathrm{Stab}_{N^{\lambda}(G)}(X)\subset N_{\Aut(X)}(G)$, that is, $\mathrm{Stab}_{N^{\lambda}(G)}(X)=N_{\Aut(X)}(G)$.
\end{proof}

This works well when $7\mid\ord(G)$ or $\rank(S)=19$. The former is presented in \cite{he2025cubicfourfoldsorder7automorphism}, and the latter gives the main result of this paper. The expressions of generators of $G$ are provided by \cite{KOIKE202512}.

\subsection{On GIT detectability}
We introduce the notion of ``detectable'' to clarify whether a cubic fourfold with larger stabilizer actually admits larger symplectic automorphism group.

We follow the notations and assumptions of \S~\ref{subsection:GITcalculation}.

\begin{defn}
    Let $G_{0}$ be the stabilizer of a generic point in $\PP\mathcal{V}_{G}^{\lambda}$ with respect to the action of $N^{\lambda}(G)$. A smooth cubic fourfold $X$ in $\PP\mathcal{V}_{G}^{\lambda}$ is said to be detectable if $G_{0}\subsetneq \mathrm{Stab}_{N^{\lambda}(G)}(X)$.
\end{defn}

For two subgroups of $\SL(6,\CC)$ with characters $(\widehat{G}_1,\lambda_1)$ and $(\widehat{G}_2,\lambda_2)$, we write $(\widehat{G}_1,\lambda_1)>(\widehat{G}_2,\lambda_2)$ if, up to conjugation, $\widehat{G}_2\subsetneq \widehat{G}_1$ and $\lambda_1|_{\widehat{G}_2}=\lambda_2$.

\begin{lem}
    \label{lem:detectibility}
Suppose $X$ is the unique smooth detectable cubic fourfold in the family $\PP\mathcal{V}_{G}^{\lambda}$ up to isomorphism. Then $(\Aut^{s}(X),\Aut(X))=(G,\mathrm{Stab}_{N^{\lambda}(G)}(X))$ if and only if there is no $X'$ such that $(\widehat{\Aut}^{s}(X'),\lambda')>(\widehat{G},\lambda)$ and $G_{0}\subsetneq N_{\Aut(X')}(G)$.

Here $\widehat{\Aut}^{s}(X')$ is the preimage of $\Aut^{s}(X')$ in $\SL(6,\CC)$, and $\lambda'$ is the natural scalar character from the action on the equation of $X'$. If such an $X'$ exists, then the uniqueness assumption implies that $X'$ is isomorphic to $X$.
\end{lem}

\begin{proof}
    By Lemma~\ref{lem:stabAut}, $(\Aut^{s}(X),\Aut(X))=(G,\mathrm{Stab}_{N^{\lambda}(G)}(X))$ is equivalent to $\Aut^{s}(X)=G$.

    If $\Aut^{s}(X)=G$ and such an $X'$ exists, then Lemma~\ref{lem:stabAut} shows that $X'$ is detectable. By uniqueness, $X'$ is isomorphic to $X$, contradicting the strict inclusion $(\widehat{\Aut}^{s}(X'),\lambda')>(\widehat{G},\lambda)$.

    Conversely, suppose that no such $X'$ exists. Since the family carries the prescribed symplectic $G$-action, we have $G\subseteq\Aut^s(X)$. If this inclusion were strict, then $(\widehat{\Aut}^{s}(X),\widehat{\lambda})>(\widehat{G},\lambda)$ and $G_{0}\subsetneq \mathrm{Stab}_{N^{\lambda}(G)}(X)=N_{\Aut(X)}(G)$ by Lemma~\ref{lem:stabAut}. Taking $X'=X$ would give a contradiction. Hence $\Aut^s(X)=G$.

\end{proof}

The above lemma works well in rank $19$ cases since the automorphism groups of cubic fourfolds $X'$ with maximal symplectic automorphism groups are known.

\subsection{Differentials and Hessian Matrix}

Differentials of the defining equation have long been used to study the automorphism group of hypersurfaces. Here we list some useful basic results that would be used to explicitly calculate automorphism groups later.

\begin{fact}\label{fact:mat_hess}
    If a polynomial $F\in \CC[x_1,\ldots,x_n]$ is preserved by $g\in \GL(n,\CC)$, then $g$ preserves the polynomial $\det\left(\left(\frac{\partial^2 F}{\partial x_i \partial x_j}\right)_{i,j}\right)$ up to some scalar.
\end{fact}
\begin{lem}\label{lem:extra_variable}
    If $H(x_1,\ldots,x_n)$ is a homogeneous polynomial of degree $d\geq 2$ which cannot be written as a polynomial in $n-1$ variables through any linear transformation and if $g=(g_{ij})\in \Hom(\CC^{n+k},\CC^n)$ satisfies $H(x)=H(g(x,y))$, then $g_{ij}=0$ for $j>n$.
\end{lem}
\begin{proof}
    Without loss of generality we can assume $k=1$. Since $H$ cannot be written as a polynomial in $n-1$ variables through any linear transformation, $(g_{ij})_{1\leq i,j\leq n}$ is invertible. By composing with its inverse, we can assume $g_{ij}=\delta_{ij}$ for $1\leq i,j\leq n$.
    
    Denote $r_i=g_{i(n+1)}$ and $r=(r_1,\ldots,r_n)^t$. Then $H(x)=H(x+yr)$. Taking derivative with respect to $y$, we deduce $\sum_i r_i \frac{\partial H}{\partial x_i}=0$. We have to show $r=0$. If $r\neq 0$, this again implies that $H$ can be written as a polynomial in $n-1$ variables through some linear transformation, leading to a contradiction.
\end{proof}

\begin{lem}\label{lem:dummy_crit}
    Let $H(x_1,\ldots,x_n)$ be a homogeneous polynomial of degree $d\geq 2$. Let $D_{d-1}(H)$ be $\Span\left\{\frac{\partial^{d-1} H}{\partial x_{i_1}\dots\partial x_{i_{d-1}}}\mid i_1,\ldots,i_{d-1}\in\{1,\ldots,n\}\right\}$. If $\dim D_{d-1}(H)=n$, then $H$ cannot be written as a polynomial in $n-1$ variables through any linear transformation.
\end{lem}
\begin{proof}
    If $H$ can be written as a polynomial in $n-1$ variables through some linear transformation, then $D_{d-1}(H)$ lies in the linear span of those variables, which is of dimension $<n$.
\end{proof}

\section{Rank 19 Case}\label{sect:rk19}

In general, assume $\rank(S)=19$ while $G$ is one of $3^{1+4}:2.2$, $A_6$, $L_2(7)$, $S_5$, $M_9$, $N_{72}$ and $T_{48}$. The case $G=L_2(7)$ is completed in \cite[Theorem 1.2]{he2025cubicfourfoldsorder7automorphism}, and the case $G=T_{48}$ is automatically completed since the group-theoretic bound is $1$.

We always assume $X$ belongs to the certain family we would be discussing. Note that any irreducible component of the moduli we will discuss is of dimension $1$.

\subsection{The Case $G=A_6$}\label{subsect:A_6}

From \cite[Section~2.5]{KOIKE202512}, there are $2$ irreducible components of the moduli. Note that the only maximal groups containing $G$ are $3^4:A_6$, $A_7$, and $M_{10}$.

\subsubsection{The First Family}
The first family consists of smooth cubic fourfolds with defining equation linearly equivalent to some element in
\[
\Span\left\{x_1^3+\cdots+x_6^3,(x_1+\cdots+x_6)(x_1^2+\cdots+x_6^2),(x_1+\cdots+x_6)^3\right\},
\]
where $G$ acts by permutations. This family admits an action of $S_6$ by permutations where $(12)$ is non-symplectic. So, the generic non-symplectic index is $2$. From \cite[Theorem 1.8]{laza2022automorphisms} or \cite[Section~8]{KOIKE202512}, one can directly check that the only smooth cubic fourfolds with maximal symplectic automorphism group lying in this moduli space are the Fermat cubic fourfold and the Clebsch--Segre cubic. Combining with the group-theoretic bounds, $(\Aut^s(X),m(X))$ must be $(A_6,2)$, $(3^4:A_6,6)$, or $(A_7,2)$.

\subsubsection{The Second Family}
For the second family, the projective representation of $A_6$ itself is not liftable. To lift $A_6$, we must pass to the central extension $3.A_{6}$, which is generated by two matrices $s$ and $t$; see \cite[Section~2.5]{KOIKE202512}.

A direct calculation shows that the eigenvalues of $s$ are $1$ with multiplicity $4$ and $-1$ with multiplicity $2$. The eigenvalues of $t$ are $1$ with multiplicity $2$, $-1$ with multiplicity $2$, and $\sqrt{-1}$ and $-\sqrt{-1}$ with multiplicity $1$. And the character $\lambda$ given by the action of $3.A_{6}$ on the cubic fourfold has value $1$ on both $s$ and $t$.

Suppose $\overline{g}\in \PGL(6,\CC)$ centralizes $A_{6}$. Take a preimage $g\in \SL(6,\CC)$ of $\overline{g}$. Then $gsg^{-1}=\mu_{1}s$ and $gtg^{-1}=\mu_{2}t$. By comparing the eigenvalues, we obtain $\mu_{1}=1$ and $\mu_{2}=\pm 1$. If $\mu_{2}=-1$, then $\lambda(gtg^{-1})=-\lambda(t)$. Such $g$ cannot have image in $C^{\lambda}(A_6)$. So, every element of $C_{\SL(6,\CC)}^{\lambda}(A_{6})$ commutes with $s$ and $t$, and hence gives an endomorphism of the corresponding $3.A_{6}$-representation. This representation is irreducible \cite[p.~9]{KOIKE202512}. By Schur's lemma, $C_{\SL(6,\CC)}^{\lambda}(A_{6})$ is the center of $\SL(6,\CC)$, and thus $C^{\lambda}(A_{6})$ is trivial.

Since $C^{\lambda}(A_{6})$ is trivial, the exact sequence
\begin{equation*}
1 \longrightarrow C^{\lambda}(A_{6}) \longrightarrow N^{\lambda}(A_{6}) \longrightarrow \Aut(A_{6})
\end{equation*}
identifies $N^{\lambda}(A_{6})$ with a subgroup of $\Aut(A_{6})$. It cannot contain $S_{6}$. Indeed, the given projective representation of $A_{6}$ would then extend to $S_{6}$. The induced $S_{6}/A_{6}\cong C_{2}$-action on $\PP\mathcal{V}_{A_{6}}^{2}\cong\PP^1$ has a fixed point, so \cite[Lemma~2.5]{KOIKE202512} would imply that this projective representation, and hence its restriction to $A_{6}$, is liftable.

By \cite[Sections~2.5 and 8(vi), Appendix~A]{KOIKE202512}, we have $M_{10}\subset N^{\lambda}(A_6)\subset\Aut(A_6)$, and $M_{10}$ has index $2$ in $\Aut(A_6)$. The exclusion of $S_6$ therefore gives $N^{\lambda}(A_6)=M_{10}$. Its action on $\PP\mathcal{V}_{A_6}^{2}$ has exactly two fixed points $X_{\pm}$, and the $M_{10}$-action at both points is symplectic. Lemma~\ref{lem:stabAut} now shows that every member of the second family with symplectic automorphism group $A_6$ has automorphism group $A_6$.

Moreover, the smooth cubic fourfold $X_{0}$ with $(\Aut^{s}(X),\Aut(X))=(A_{7},A_{7})$ lies in the second family.

\subsection{The Case $G=S_5$} 
\label{subsection:rank19S5}
From \cite[Section~7.3]{KOIKE202512}, there are $2$ irreducible components of the moduli.
\subsubsection{The First Family}
The first family consists of smooth cubic fourfolds with defining equation linearly equivalent to some element in
\[
\mathcal{V}^{1}_{S_{5}}=\Span\left\{y^2 s_1(x), s_1(x)^3, s_1(x)s_2(x), s_3(x)\right\}.
\]

This family corresponds to a liftable projective representation, hence to a linear representation $\rho_{1}$. We use $S_{5}$ instead of $\rho_{1}(S_{5})$ as a subgroup of $\GL(6,\CC)$.

We will use projective coordinates under the above basis to represent a cubic fourfold in this family with variables $x_1,\ldots,x_5,y$, where $s_k(x)=\sum_i x_i^k$. The group $G$ acts by permutation on $\{x_i\}$ and by the sign character on $y$. By GIT calculation, for $X$ in this moduli space, $(\Aut^s(X),m(X))$ equals $(S_5,2)$ or $(A_{3,5},6)$.

The automorphism group of $S_{5}$ is $S_{5}$ itself. By Equation~\eqref{equation:computeNlambda}, we know that $N(S_{5})=\langle S_{5}, C(S_{5})\rangle$.

Then $C(S_{5})$ is given by matrices of the following form:
\begin{align*}
    M(a,b,g)=\begin{pmatrix}
    a&b&b&b&b&0 \\
    b&a&b&b&b&0 \\
    b&b&a&b&b&0 \\
    b&b&b&a&b&0 \\
    b&b&b&b&a&0 \\
    0&0&0&0&0&g
\end{pmatrix},
\end{align*}

which act on $\PP\mathcal{V}^{1}_{S_{5}}$ by the matrices:
\begin{align*}
 \begin{pmatrix}
    g^{2}(a+4b)&0&0&0\\
    0&(a+4b)^{3}&b(2a+3b)(a+4b)&b^{2}(3a+2b)\\
    0&0&(a+4b)(a-b)^{2}&3b(a-b)^{2}\\
    0&0&0&(a-b)^{3}
\end{pmatrix}.   
\end{align*}

All cubic fourfolds in $\PP\mathcal{V}^{1}_{S_{5}}$ admit a non-symplectic automorphism of order 2 given by $\frac{1}{2}(0,0,0,0,0,1)$.

Cubic fourfolds $X=[1,c,-\frac{5}{3}c,cv]$ with $c\neq 0$ and $v\neq \frac{2}{25}$ are detectable, and $\mathrm{Stab}_{N(S_{5})}(X)/S_{5}\cong C_{6}$. The generator of $C_{6}$ is given by $M(\frac{4+\omega}{\omega-1},1,-1)$, which has order $6$ up to a scalar.

They form a single orbit and they are the only detectable cubic fourfolds in this family.

Note that $S_{5}$ is a subgroup of $A_{3,5}$, both as finite subgroups of $\SL(6,\CC)$ up to conjugation. And the cubic fourfold $X^{\prime}$ with $\Aut^{s}(X^{\prime})=A_{3,5}$ lies inside $\PP\mathcal{V}^{1}_{S_{5}}$.

We claim that $X^{\prime}$ satisfies the criterion in Lemma~\ref{lem:detectibility} and hence is the unique detectable member of the family.

The equation of $X^{\prime}$ is given by

\[
\begin{aligned}
x_{1}^{3}+x_{2}^{3}+x_{3}^{3}+x_{4}^{3}+x_{5}^{3}+x_{6}^{3}+x_{7}^{3}+x_{8}^{3}&=0,\\
x_{1}+x_{2}+x_{3}&=0,\\
x_{4}+x_{5}+x_{6}+x_{7}+x_{8}&=0,
\end{aligned}
\]

and $A_{3,5}$ acts on this cubic fourfold as a subgroup of $S_{8}$ with respect to the natural representation.

Moreover,
\[
\Aut(X^{\prime})=\left\langle A_{3,5}, \frac{1}{3}(1,1,1,0,0,0,0,0),(12)\right\rangle,
\]
and
\[
N_{\Aut(X^{\prime})}(S_{5})/S_{5}\cong \left\langle \frac{1}{3}(1,1,1,0,0,0,0,0),(12)\right\rangle\cong C_{6}.
\]
Thus $X^{\prime}$ is detectable.

So for this family, the generic non-symplectic index is $2$.

\subsubsection{The Second Family} The second family consists of smooth cubic fourfolds with defining equation linearly equivalent to some element in
\[
\mathcal{V}_{S_{5}}^{2}=\Span\left\{s_{1}(x)^{3},s_{1}(x)s_{2}(x),I_{1}-I_{2}\right\}.
\]

More explicit equations can be found in \cite[(7.6), Page~27]{KOIKE202512}. We use coordinates under this basis to represent a cubic fourfold in this family.

This corresponds to another representation $\rho_{2}$ of $S_{5}$, and we write $S_{5}$ instead of $\rho_{2}(S_{5})$ for simplicity.

We know $N(S_{5})=\langle C(S_{5}),S_{5}\rangle$ and $C(S_{5})$ consists of matrices of the following form:

\begin{align*}
A(a,b)=\begin{pmatrix}
        a&b&b&b&b&b\\
        b&a&b&b&b&b\\
        b&b&a&b&b&b\\
        b&b&b&a&b&b\\
        b&b&b&b&a&b\\
        b&b&b&b&b&a
    \end{pmatrix}.
\end{align*}

And these matrices act on $\mathcal{V}_{S_{5}}^{2}$ via
\begin{align*}
    \begin{pmatrix}
        (a+5b)^{3}&2(a+2b)(a+5b)b&0\\
        0&(a-b)^{2}(a+5b)&0\\
        0&0&(a-b)^{3}
    \end{pmatrix}.
\end{align*}

For a generic $X$ in this family,
$\mathrm{Stab}_{N(S_{5})}(X)=S_{5}$. If $X=[t,0,1]$ with $t\neq 0$,
then $\mathrm{Stab}_{N(S_{5})}(X)/S_{5}\cong C_{3}$, generated by the
image of $A(1,\frac{\omega-1}{5+\omega})$. The displayed action of
$C(S_5)$ shows that all points $[t,0,1]$ with $t\neq0$ lie in one orbit.
Its representative $[1,0,1]$ is singular at
$[5,-1,-1,-1,-1,-1]$. Therefore every smooth member of this family with
symplectic automorphism group $S_5$ has non-symplectic index $1$.

\subsubsection{Lattice Theoretical Approach} These arguments on detectable cubic fourfolds can be interpreted purely lattice theoretically as follows:

Let the lattice $\Lambda$ be the cohomology of cubic fourfolds $H^{4}(X,\ZZ)$ equipped with the intersection pairing. Let $\Lambda_{0}$ be $\langle h^{2}\rangle^{\perp}$, where $h^2$ is the square of the hyperplane class. Let $S_{1}$ and $S_{2}$ be the coinvariant lattices of the actions of $S_{5}$ and $A_{3,5}$ on $\Lambda_{0}$, with $T_{1}$ and $T_{2}$ being their orthogonal complements, respectively.

Then $\rank(S_{1})=19$, $\rank(S_{2})=20$, $q_{S_{1}}=4_{5}^{-1}3^{-1}5^{-2}$, and $q_{S_{2}}=3^{-2}5^{-2}$; see \cite[Table~1, items 82 and 128]{hohn2016290}. Moreover, $T_{2}=A_{2}(-5)$ and $q_{T_{2}}=3^{-1}5^{-2}$ by \cite[Theorem 1.8(6)]{laza2022automorphisms}.

For the first family, $S_{1}\subset S_{2}$ and $T_{2}\subset T_{1}$, with
$q_{T_1}=4_3^{-1}5^{-2}$. A direct calculation with the primitive
embedding $T_2\subset T_1$ shows that the standard order-$3$ isometry of
$T_2=A_2(-5)$, extended trivially on its orthogonal complement and on
$\langle h^2\rangle$, preserves the gluing subgroup defining $T_1$.
It therefore extends to an isometry $f$ of
$T_1\oplus\langle h^2\rangle$ by
\cite[Proposition~1.4.1]{Nikulin_1980}. By
\cite[Table~1, item~82]{hohn2016290}, the homomorphism
$O(S_1)\to O(q_{S_1})$ is surjective. Hence one can choose
$g\in O(S_1)$ so that $g$ and $f$ glue to an isometry
$\Phi\in O(\Lambda)$ preserving $(S_1,T_1)$ and $(S_2,T_2)$.

Let $X'$ be the cubic fourfold with symplectic automorphism group
$A_{3,5}$ described above. Its coinvariant and invariant lattices are
$S_2$ and $T_2$. The eigenspaces of the order-$3$ isometry are the two
isotropic lines in $T_{2,\CC}$, namely $H^{3,1}(X')$ and
$H^{1,3}(X')$. Thus $\Phi$ is a Hodge isometry fixing $h^2$. The global
Torelli theorem gives a non-symplectic automorphism of $X'$ realizing
$\Phi$; it normalizes $S_5$ by
Lemma~\ref{Lemma:normalize_latticecriterion}. Together with the explicit
normalizer calculation above, this shows that $X'$ is detectable.

\subsection{The Case $G=M_9$}\label{subsect:M_9}
From \cite[Section~7.4]{KOIKE202512}, there is exactly one irreducible component of the moduli.

Consider $X_{15}^{\prime}$ in \cite[Section~6.1]{yang2024automorphism}. One can directly calculate that $\Aut^s(X_{15}^{\prime})\cong G$ and its non-symplectic index is $3$. Combining with the group-theoretic bound and Corollary~\ref{cor:gen12}, all possible indices are $1$ and $3$.

To be more explicit, $N(M_{9})=\langle M_{9}, E\rangle$. The matrix $E\in \PGL(6,\CC)$ is given by
\[
    \begin{pmatrix}
        F&cF\\
        bF&dF
\end{pmatrix},
\]
where $b=-\omega$, $c=\frac{(\omega-1)\sqrt{3}}{3}$, $d=-\sqrt{-1}$ and 
\[F=\begin{pmatrix}
    \omega & 1& 1\\
    1& \omega & 1\\
    1 & 1 & \omega  
\end{pmatrix}.\] 

Direct computation shows that $N(M_{9})/M_{9}=\langle E\rangle\cong C_{3}$.

The $M_{9}$-invariant family $\PP\mathcal{V}_{M_{9}}$ is the projectivization of the span of the following equations from \cite[Equation~(7.9)]{KOIKE202512}:

\[
\begin{aligned}
    &x_{1}^3+x_{2}^{3}+x_{3}^{3}-3(1-\sqrt{3})x_{1}x_{2}x_{3}+(\omega+\sqrt{-1}\omega^{2})(y_1^3+y_{2}^{3}+y_{3}^{3}-3(1+\sqrt{3})y_{1}y_{2}y_{3});\\
    &(\sqrt{-1}-\omega^2)(x_{1}y_{1}^{2}+x_{2}y_{2}^{2}+x_{3}y_{3}^{2}-(1-\sqrt{3})(x_{1}y_{2}y_{3}+y_{1}x_{2}y_{3}+y_{1}y_{2}x_{3}))\\
    &\qquad+x_{1}^{2}y_{1}+x_{2}^{2}y_{2}+x_{3}^{2}y_{3}-(1+\sqrt{3})(y_{1}x_{2}x_{3}+x_{1}y_{2}x_{3}+x_{1}x_{2}y_{3}).
\end{aligned}
\]

And $E$ acts on this family via the matrix

\[
\begin{pmatrix}
    \frac{1}{2}(1+3\sqrt{-1})(3+\sqrt{-3})& \frac{1}{2}((-9+9\sqrt{-1})+(3-9\sqrt{-1})\sqrt{3})\\
    \sqrt{-1}(3+\sqrt{3})&\frac{1}{2}((3+3\sqrt{-1})-(1-\sqrt{-1})\sqrt{3})
\end{pmatrix},
\]
whose eigenvectors are $(1,c)$ and $(1,-c)$, where
\[
c=\frac{\sqrt{3}}{2}(\sqrt{3}-2-\sqrt{-1})
=\frac{3-\sqrt{3}}{\sqrt{2}}e^{-\frac{7\pi}{12}\sqrt{-1}}.
\]

This defines two cubic fourfolds. The first cubic fourfold corresponding to $(1,c)$ is singular, with a $2$-dimensional singular locus containing the point $[0,0,\omega-\sqrt{-1},0,0,1]$. The second cubic fourfold corresponding to $(1,-c)$ is smooth.

The family also contains the cubic fourfold with symplectic automorphism
group $3^{1+4}:2.2^{2}$; see
Section~\ref{subsection:incidence-relations}.

\subsection{The Case $G=N_{72}$} 
\label{subsection:thecaseN72}
\subsubsection{GIT Method}
From \cite[Section~7.5]{KOIKE202512}, there is exactly one irreducible component of the moduli, consisting of smooth cubic fourfolds with defining equations linearly equivalent to some elements in

\[
\begin{aligned}
\mathcal{V}_{N_{72}}=\Span\{&s_1(y)^3,s_1(x)^2 s_1(y),s_1(y)(s_2(x)+s_2(y)),\\
&2s_1(x)^3-9s_1(x)s_2(x)+9s_3(x)+9s_2(y)s_1(y)-9s_3(y)\},
\end{aligned}
\]
with variables $x_1,x_2,x_3,y_1,y_2,y_3$, where $s_k(x)=\sum_i x_i^k$ and $s_k(y)=\sum_i y_i^k$, and a generic linear combination in these $4$ terms gives a smooth cubic fourfold. The whole family admits the non-symplectic automorphism given by the permutation $(12)$.

Similarly as before, we use projective coordinates $[v_1,v_2,v_3,v_4]$ with respect to such a basis to represent a cubic fourfold in this family.

From Corollary~\ref{cor:gen12}, generic smooth $X$ expressed in this form has non-symplectic index $2$. Moreover, GIT calculation implies that any smooth cubic fourfold in the family has $(\Aut^s(X),m(X))$ being $(N_{72},2)$ or $(3^4:A_6,6)$. The latter case corresponds to the Fermat cubic fourfold as a single point in the moduli.

The group $N_{72}$ is centerless, with
$\operatorname{Out}(N_{72})\cong C_2$ and
$N_{72}^{\mathrm{ab}}\cong C_2^2$. For the chosen representation, the
nontrivial outer automorphism does not preserve the character of the chosen
six-dimensional representation of its linear inverse image: it interchanges
two order-$6$ classes on which that character has values $2$ and $-1$.
Moreover, $N_{72}^{\mathrm{ab}}$ has no nontrivial $\mu_3$-valued character,
so no scalar twist can preserve the representation. Thus the nontrivial outer
automorphism is not induced by an element of $N^\lambda(N_{72})$. Consequently,
\[
N^{\lambda}(N_{72})=\langle C^{\lambda}(N_{72}),N_{72}\rangle.
\]
Conjugation by $(12)$ agrees with conjugation by
$(12)(56)\in N_{72}$. Thus $(56)$ centralizes $N_{72}$ projectively, and
$\langle N_{72},(12)\rangle\cong N_{72}\times C_2$.

The matrices of $C^{\lambda}(N_{72})$ are of the following form:

\begin{align*}
    H(a,b,\lambda)=\begin{pmatrix}
        a+\lambda & a & a & 0 & 0 & 0\\
        a & a+\lambda & a & 0 & 0 & 0\\
        a & a & a+\lambda & 0 & 0 & 0\\
        0 & 0 & 0 & b+\lambda & b & b\\
        0 & 0 & 0 & b & b+\lambda & b\\
        0 & 0 & 0 & b & b & b+\lambda
    \end{pmatrix}.
\end{align*}

These matrices act on $\mathcal{V}_{N_{72}}$ via the matrix:

\begin{align*}
    \begin{pmatrix}
        (3b+\lambda)^{3} & 0 & b(3b+2\lambda)(3b+\lambda) & 18b(3b^2+3b\lambda+\lambda^{2})\\
        0 & (3a+\lambda)^{2}(3b+\lambda) & a(3a+2\lambda)(3b+\lambda)& 0 \\
        0 & 0 & (3b+\lambda)\lambda^{2} & 0\\
        0 & 0 & 0 & \lambda^{3}
    \end{pmatrix}.
\end{align*}

And the matrix $(12)$ acts on $\mathcal{V}_{N_{72}}$ trivially. Therefore, every cubic fourfold in this family admits a symmetry given by $(12)$, which is non-symplectic.

The only detectable cubic fourfolds are given by $[p,q,0,1]$ with $q\neq 0,-2$. They are smooth cubic fourfolds that form a single orbit. For each of these cubic fourfolds $X$, $\mathrm{Stab}_{N^{\lambda}(N_{72})}(X)/N_{72}\cong C_{6}$, whose generator is given by $H(\frac{-\omega-1}{3},\frac{\omega-1}{3},1)$. 

Up to conjugation in $\GL(6,\CC)$, the linear lift of $N_{72}$ is contained in that of $3^{4}:A_{6}$. Thus the Fermat cubic fourfold $X^{\prime}$ with $\Aut^{s}(X')=3^{4}:A_{6}$ lies in this family. After a linear change of coordinates, we may assume that $3^{4}:A_{6}$ acts on $\CC^{6}$ by coordinate permutations and by the diagonal matrices $\frac{1}{3}(1,0,\ldots,2,\ldots,0)$, where $2$ occurs in the $i$-th position for $i=2,3,4,5,6$. The full automorphism group of the Fermat cubic fourfold is $\Aut(X^{\prime})=3^{5}:S_{6}$. We claim that $X^{\prime}$ satisfies the criterion in Lemma~\ref{lem:detectibility} and is therefore a representative of the unique detectable isomorphism class.

In fact, after a conjugation, $N_{72}$ is generated by $(12)(56)$, $(34)(56)$, $(13)(24)$, $\frac{1}{3}(1,2,0,0,0,0)$, and $\frac{1}{3}(0,0,1,2,0,0)$. And
\[
N_{\Aut(X^{\prime})}(N_{72})/N_{72}=\left\langle (12), \frac{1}{3}(1,0,0,0,0,0)\right\rangle\cong C_{6}.
\]

So we obtain that the non-symplectic index of this $N_{72}$ family is $2$.

\subsubsection{Lattice Theoretical Approach}
The fact that the Fermat cubic fourfold is detectable can also be interpreted lattice theoretically.

Let $S_{1}$ and $S_{2}$ be the coinvariant lattices of $N_{72}$ and
$3^{4}:A_{6}$ in $\Lambda_{0}$, and let $T_{1}$ and $T_{2}$ be their
respective orthogonal complements. Then $S_{1}\subset S_{2}$,
$T_{2}\subset T_{1}$, $\rank(S_{1})=19$, and $\rank(S_{2})=20$.
Moreover, $q_{S_{1}}=4_{7}^{+1}3^{+2}9^{+1}$,
$q_{S_{2}}=3^{+2}9^{-1}$,
$q_{T_{1}}=4_{1}^{+1}3^{+1}9^{-1}$, and
$q_{T_{2}}=3^{+1}9^{+1}$; see
\cite[Table~1, items~85 and 101]{hohn2016290} and
\cite[Theorem~1.8(1), Section~4.5]{laza2022automorphisms}.

The lattice $T_{2}=A_{2}(-3)$ admits an order-$3$ automorphism preserving
the gluing subgroup that defines $T_{1}$. Extending it trivially on the
rank-one orthogonal complement and on $\langle h^2\rangle$ gives an
isometry $f$ of $T_{1}\oplus\langle h^2\rangle$. The homomorphism
$O(S_{1})\rightarrow O(q_{S_{1}})$ is surjective by
\cite[Table~1, item~85]{hohn2016290}. Hence one can choose
$g\in O(S_{1})$ such that $g$ and $f$ glue to an isometry
$\Phi\in O(\Lambda)$ preserving $(S_{1},T_{1})$ and $(S_{2},T_{2})$.

Let $X'$ be the Fermat cubic fourfold above. Its coinvariant and invariant
lattices are $S_{2}$ and $T_{2}$, and the two eigenspaces of the order-$3$
automorphism of $T_{2}$ are $H^{3,1}(X')$ and $H^{1,3}(X')$.
Thus $\Phi$ is a Hodge isometry fixing $h^2$. The global Torelli theorem
gives a non-symplectic automorphism $\Psi$ of $X'$ realizing $\Phi$.
By construction, $\Psi$ preserves $S_{1}$ and $T_{1}$ and normalizes
$G_{1}=N_{72}$ by Lemma~\ref{Lemma:normalize_latticecriterion}.

\subsection{The Case $G=3^{1+4}:2.2$}
From \cite[Section~7.1]{KOIKE202512}, any $X$ with $\Aut^s(X)\supset G$ is given, after a linear change of coordinates, by
\[
    x_1^3+x_2^3+x_3^3+\lambda x_1 x_2 x_3 +x_4^3+x_5^3+x_6^3+\lambda x_4 x_5 x_6=0.
\]

We use $[c_1,c_2]$ to represent a cubic fourfold with equation given by $c_1(x_{1}^{3}+x_{2}^{3}+x_{3}^{3}+x_{4}^{3}+x_{5}^{3}+x_{6}^{3})+c_2(x_{1}x_{2}x_{3}+x_{4}x_{5}x_{6})=0$.

Let $\PP \mathcal{V}_{3^{1+4}:2.2}$ consist of the equations of these cubic fourfolds. 

Note that $(12)$ is a non-symplectic action on the whole family. So, there is exactly one connected family with generic index $2$.

\subsubsection{Hessian Method}
\begin{prop}\label{prop:3^1+4:2.2}
    In this case,\footnote{Note that $3^{1+4}:2.2$ is isomorphic to some semidirect product $3^{1+4}:4$, which is naturally a normal subgroup of another semidirect product $3^{1+4}:D_8$ through the inclusion $C_4\subset D_8$.}
    {\footnotesize
    \begin{align*}
    &(\Aut^s(X),\Aut(X))\\
    =&\begin{cases}
    (3^4:A_6,3^5:S_6), & \text{if }\lambda\in\{0,6,6\omega,6\omega^2\}\\
    (((3\times(3^2:3)):3):Q_8, ((3\times(3^2:3)):3):(4^2:2)),
       & \text{if }\lambda\in\{3(\sqrt{3}\pm 1),3\omega(\sqrt{3}\pm 1),3\omega^2(\sqrt{3}\pm 1)\}\\
    (3^{1+4}:2.2, 3^{1+4}:D_8), & \text{otherwise}.
    \end{cases}
    \end{align*}
    }
    Note that the equation gives a smooth cubic fourfold iff $\lambda\notin\{-3,-3\omega,-3\omega^2\}$.
\end{prop}
\begin{proof}
    For the calculation of the linear automorphism group of the elliptic curve $C_\lambda$ given by $x_1^3+x_2^3+x_3^3+\lambda x_1 x_2 x_3=0$, see \cite[Lemma 3.12]{yang2024automorphism}. Precisely, it is divided into $3$ cases:
    \begin{enumerate}
        \item $\lambda\in\{0,6,6\omega,6\omega^2\}$. In these cases $C_\lambda$ has $j$-invariant $0$ and linear automorphism group $3^3:S_3$.
        \item $\lambda\in\{3(\sqrt{3}\pm 1),3\omega(\sqrt{3}\pm 1),3\omega^2(\sqrt{3}\pm 1)\}$. In these cases $C_\lambda$ has $j$-invariant $1728$ and linear automorphism group $(3^2:3):4$.
        \item Otherwise, $C_\lambda$ has $j$-invariant $\neq 0,1728$ and linear automorphism group $3^2:S_3$.
    \end{enumerate}
    The following fact shows that the automorphism group of a cubic fourfold $X$ with $\Aut^{s}(X)\supset3^{1+4}:2.2$ is determined by the linear automorphism groups of the elliptic curves in its equation.
    \begin{claim}
        Let
        \[
        F(x_1,\ldots,x_6)=x_1^3+x_2^3+x_3^3+\lambda x_1 x_2 x_3
        +x_4^3+x_5^3+x_6^3+\lambda x_4 x_5 x_6.
        \]
        Suppose that $F$ defines a smooth cubic fourfold and that $\lambda\notin \{0,6,6\omega,6\omega^2\}$. If $g\in \GL(6,\CC)$ preserves $F$ up to a scalar, then $g$ has the form $\left[\begin{smallmatrix}
            A_1 & 0\\
            0 & A_2
        \end{smallmatrix}\right]$ or $\left[\begin{smallmatrix}
            0 & A_1\\
            A_2 & 0
        \end{smallmatrix}\right]$ of $3\times 3$ blocks.
    \end{claim}
    \begin{proof}
        By Fact~\ref{fact:mat_hess} and a direct calculation, $g$ preserves $H(x_1,x_2,x_3)H(x_4,x_5,x_6)$ up to a scalar, where
        \[
        H(x_1,x_2,x_3)=x_1^3+x_2^3+x_3^3
        -\left(\frac{\lambda}{3}+\frac{36}{\lambda^2}\right)x_1x_2x_3.
        \]
        Since $\lambda\notin \{0,6,6\omega,6\omega^2\}$, the polynomial $H$ is irreducible. By unique factorization, we may assume, after exchanging the two factors if necessary, that $g$ preserves $H(x_1,x_2,x_3)$ up to a scalar. This polynomial defines a smooth hypersurface, so Lemma~\ref{lem:extra_variable} gives the desired block form.
    \end{proof}
    Once the automorphism group is determined, one can directly calculate the symplectic automorphism group, and the proof is finished. Note that the first case is just the Fermat cubic fourfold and the second case is just $X_7^{\prime}$ in \cite[Section~6.1]{yang2024automorphism}. See also \cite[Table~1]{yang2024automorphism}.
\end{proof}

\subsubsection{GIT Approach}
One can also directly compute that

\[N(3^{1+4}:2.2)/(3^{1+4}:2.2)=\left\langle \frac{1}{3}(0,1,0,0,1,0), (12),\diag(A,A)\right\rangle\cong C_{2}:A_{4}.\]

Here $A$ is the Vandermonde matrix given by
\[
    \begin{pmatrix}
        1&1&1\\
        1&\omega &\omega^{2}\\
        1&\omega^{2}&\omega
    \end{pmatrix}.
\]

The generator $(12)$ fixes every cubic fourfold in this family. It is a non-symplectic action. The generator $\frac{1}{3}(0,1,0,0,1,0)$ acts by $b=\frac{1}{3}(0,1)$. And $\diag(A,A)$ acts by
\[
a=\begin{pmatrix}
    3&1\\
    18&-3
\end{pmatrix}.
\]

The detectable cubic fourfolds are fixed by either an order-$2$ or an order-$3$ element, corresponding to Fermat cubic fourfolds $X^{\prime}$ and cubic fourfolds $X_{1}$ with maximal symplectic group $3^{1+4}:2.2^{2}$, respectively. Explicitly, Fermat cubic fourfolds are given by equations $[1,0]$, $[1,6]$, $[1,6\omega]$, and $[1,6\omega^{2}]$, fixed by $b$, $bab$, $ab$, and $ab^2$. These Fermat cubic fourfolds form a single orbit. And these cubic fourfolds with symplectic group $3^{1+4}:2.2^{2}$ are given by $[1,3(\sqrt{3}\pm1)]$, $[1,3\omega(\sqrt{3}\pm1)]$, and $[1,3\omega^{2}(\sqrt{3}\pm1)]$, fixed by $a$, $bab^2$, and $b^{2}ab$. They also form a single orbit. Therefore, the non-symplectic index of $\PP\mathcal{V}_{3^{1+4}:2.2}$ is $m=2$.

\subsubsection{Lattice Theoretical Approach}
Repeating the argument of Section~\ref{subsection:thecaseN72}, the order-$3$
automorphism of the $3^{4}:A_{6}$-invariant lattice
$T_{2}=A_{2}(-3)$ extends to an order-$3$ non-symplectic automorphism of
$X^{\prime}$ that normalizes $3^{1+4}:2.2$. This confirms that
$X^{\prime}$ is detectable.

\subsection{Incidence Relations}\label{subsection:incidence-relations}
Observe that cubic fourfolds $X^{\prime}$ with maximal symplectic automorphism group can be found in many different $1$-dimensional moduli spaces of cubic fourfolds with symplectic automorphism groups whose coinvariant lattices have rank $19$. One can recognize this phenomenon as the intersection of these moduli spaces at the points representing cubic fourfolds with maximal symplectic automorphism group. In this section, we aim to detect all such incidence relations between these $1$-dimensional families.

\begin{defn}
    Let $X$ be a smooth cubic fourfold, and let $\mathcal{F}_{1}$ and $\mathcal{F}_{2}$ be two distinct irreducible families among the rank-$19$ cases, with generic symplectic automorphism groups $G_{1}$ and $G_{2}$. We say that $\mathcal{F}_{1}$ and $\mathcal{F}_{2}$ intersect at $X$ if, up to conjugation in $\PGL(6,\CC)$, both $G_{1}$ and $G_{2}$ are subgroups of $\Aut^{s}(X)$.
\end{defn}

Clearly, if $X$ is an intersection of two rank-$19$ families $\mathcal{F}_{1}$ and $\mathcal{F}_{2}$, then $X$ must admit a maximal symplectic automorphism group. Therefore, the incidence relations are actually the inclusion relations between the rank-$19$ and rank-$20$ symplectic automorphism groups in $\PGL(6,\CC)$.

We first list the inclusion relations as abstract groups to exclude most cases. We form the diagram as follows.

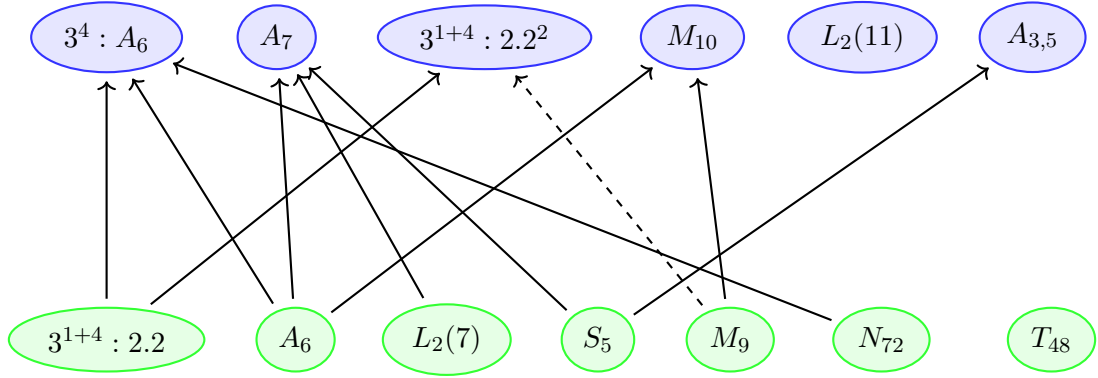
\begin{center}
\begin{tikzpicture}[
    biggroup/.style={
        ellipse,
        draw=blue!80,
        thick,
        fill=blue!10,
        minimum width=1.8cm,  
        minimum height=0.8cm,
        align=center,
        font=\small
    },
    smallgroup/.style={
        ellipse,
        draw=green!80,
        thick,
        fill=green!10,
        minimum width=1.5cm,  
        minimum height=0.7cm,
        align=center,
        font=\small
    },
    connection/.style={
        ->,
        thick,
        black,
        shorten >=3pt,
        shorten <=3pt
    }
]

\node[biggroup] (B1) at (-7, 3) {$3^4:A_6$};
\node[biggroup, minimum width=0.9cm] (B2) at (-4.75, 3) {$A_7$};
\node[biggroup, minimum width=2.1cm] (B3) at (-2, 3) {$3^{1+4}:2.2^2$}; 
\node[biggroup, minimum width=0.9cm] (B4) at (0.75, 3) {$M_{10}$};
\node[biggroup, minimum width=1.5cm] (B5) at (3, 3) {$L_2(11)$};
\node[biggroup, minimum width=1.4cm] (B6) at (5.25, 3) {$A_{3,5}$};

\node[smallgroup] (S1) at (-7, -1) {$3^{1+4}:2.2$};
\node[smallgroup, minimum width=0.9cm] (S2) at (-4.5, -1) {$A_6$};
\node[smallgroup] (S3) at (-2.5, -1) {$L_2(7)$};
\node[smallgroup, minimum width=0.9cm] (S4) at (-0.5, -1) {$S_5$};
\node[smallgroup, minimum width=0.9cm] (S5) at (1.25, -1) {$M_9$};
\node[smallgroup, minimum width=0.9cm] (S6) at (3.25, -1) {$N_{72}$};
\node[smallgroup, minimum width=0.9cm] (S7) at (5.5, -1) {$T_{48}$};

\foreach \i/\j in {1/1,1/2,1/6,2/2,2/3,2/4,3/1,4/2,4/5,6/4} {
    \draw[connection] (S\j) -- (B\i);
}
\draw[connection] (S5) -- (B3);
\end{tikzpicture}
\end{center}

Then we check their incidence relations at the level of linear or projective representations case by case. If $G_{1}<G_{2}$ as abstract groups, we must verify that the representation of $G_{1}$ obtained by restriction from $G_{2}$ is projectively equivalent to the prescribed $G_1$-representation. This can often be decided by comparing characters. For non-liftable representations, we work with their central $\mu_{3}$-extensions in $\GL(6,\CC)$ and, when necessary, construct an explicit intertwining matrix. We refer to \cite[Sections~7 and 8]{KOIKE202512} for these representations.

Let $A_{7}^{1}$ and $A_{7}^{2}$ denote the two projective representations of $A_{7}$ corresponding to cubic fourfolds with automorphism groups $A_{7}$ and $S_{7}$, respectively. Let $A_{6}^{1}$ (respectively, $S_{5}^{1}$) and $A_{6}^{2}$ (respectively, $S_{5}^{2}$) denote the two representations of $A_{6}$ (respectively, $S_{5}$) associated with the rank-$19$ families whose generic non-symplectic indices are $2$ and $1$. For each remaining group $G$ in the rank-$19$ or rank-$20$ cases, there is a unique projective representation admitting a symplectic action on a smooth cubic fourfold. For simplicity, we also write $G$ for its image in $\PGL(6,\CC)$.

Except for the natural inclusion relations such as $M_{9}<M_{10}$ and the cases already argued in this section, the remaining cases are as follows. Each inclusion of representations also gives the corresponding incidence relation between families.
\begin{enumerate}
    \item $L_{2}(7)<A_{7}^{1}$ and $L_{2}(7)<A_{7}^{2}$; see \cite[Section~4.2]{he2025cubicfourfoldsorder7automorphism}.
    \item $S_{5}^{2}<A_{7}^{1}$ and $S_{5}^{1}<A_{7}^{2}$ by directly comparing the characters; see \cite[Sections~2 and 7]{KOIKE202512}.
    \item $S_{5}^{1}$ is not contained in $A_{7}^{1}$ because every smooth cubic fourfold $X$ with $\Aut^{s}(X)=S_{5}^{1}$ admits an extra non-symplectic action of order $2$, while the smooth cubic fourfold $X^{\prime}$ with $\Aut^{s}(X^{\prime})=A_{7}^{1}$ admits no non-symplectic action. The representation $S_{5}^{2}$ is not contained in $A_{7}^{2}$; this also follows from a direct comparison of characters.
    \item For these projective representations one has
    \[
        M_{9}<3^{1+4}:2.2^{2}.
    \]
    This inclusion is verified directly, see
    \path{gap_classification/gap_fourfold_cross_dimension/result/gap_fourfold_cross_dimension_positive_edges.g}
    \cite{FuWangZhengAutcub4fold2026}.
    \item For the inclusion $A_6^1<3^4:A_6$, see Section~\ref{subsect:A_6}.
    The inclusion $A_6^2<3^4:A_6$ is verified in
    \path{gap_classification/gap_fourfold_cross_dimension/result/gap_fourfold_cross_dimension_positive_edges.g}
    \cite{FuWangZhengAutcub4fold2026}.
\end{enumerate}

The following diagram shows incidence relations between moduli spaces for the cases $\rank(S)=19,20$. The $G$ and $m$ in the notation $(G,m)$ mean the symplectic automorphism group and the non-symplectic index, respectively. Those open nodes represent cubic fourfolds of rank $20$, while those solid nodes represent special members of rank-$19$ families whose non-symplectic index is strictly larger than the generic index.

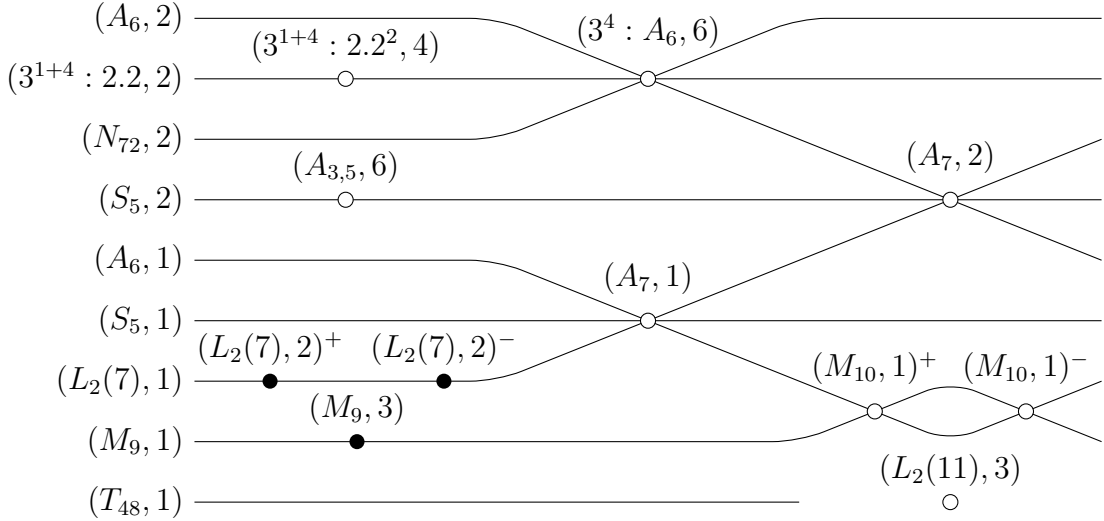
\begin{center}
\begin{tikzpicture}
\draw[rounded corners=10pt] (2,0) -- (6,0) -- (14,-3.2);
\draw[rounded corners=10pt] (2,-1.6) -- (6,-1.6) -- (10,0) -- (14,0);
\draw[rounded corners=10pt] (2,-3.2) -- (6,-3.2) -- (12,-5.6) -- (14,-4.8);
\draw[rounded corners=10pt] (2,-4.8) -- (6,-4.8) -- (14,-1.6);
\draw[rounded corners=10pt] (2,-5.6) -- (10,-5.6) -- (12,-4.8) -- (14,-5.6);

\draw (2,-0.8) -- (14,-0.8);
\draw (2,-2.4) -- (14,-2.4);
\draw (2,-4.0) -- (14,-4.0);
\draw (2,-6.4) -- (10,-6.4);

\node[left] at (2,0) {$(A_6,2)$};
\node[left] at (2,-0.8) {$(3^{1+4}:2.2,2)$};
\node[left] at (2,-1.6) {$(N_{72},2)$};
\node[left] at (2,-2.4) {$(S_5,2)$};
\node[left] at (2,-3.2) {$(A_6,1)$};
\node[left] at (2,-4.0) {$(S_5,1)$};
\node[left] at (2,-4.8) {$(L_2(7),1)$};
\node[left] at (2,-5.6) {$(M_9,1)$};
\node[left] at (2,-6.4) {$(T_{48},1)$};
\node[circle, draw=black, fill=white, inner sep=2pt, label={[label distance=4.5pt]above:{$(3^{4}:A_6,6)^\dagger $}}] at (8,-0.8) {};
\node[circle, draw=black, fill=white, inner sep=2pt, label={[label distance=-1pt]above:{$(3^{4}:A_6,6)^\dagger $}}] at (4,-3.2) {};
\node[circle, draw=black, fill=white, inner sep=2pt, label={[label distance=3pt]above:{$(A_7,1) $}}] at (8,-4.0) {};
\node[circle, draw=black, fill=white, inner sep=2pt, label={[label distance=3pt]above:{$(A_7,2) $}}] at (12,-2.4) {};
\node[circle, draw=black, fill=white, inner sep=2pt, label={[label distance=3pt]above:{$(M_{10},1)^+ $}}] at (11,-5.2) {};
\node[circle, draw=black, fill=white, inner sep=2pt, label={[label distance=3pt]above:{$(M_{10},1)^- $}}] at (13,-5.2) {};
\node[circle, draw=black, fill=white, inner sep=2pt, label={[label distance=-1pt]above:{$(3^{1+4}:2.2^2,4)^{\dagger}$}}] at (4,-0.8) {};
\node[circle, draw=black, fill=white, inner sep=2pt, label={[label distance=-1pt]above:{$(3^{1+4}:2.2^2,4)^{\dagger}$}}] at (7,-5.6) {};
\node[circle, draw=black, fill=white, inner sep=2pt, label={[label distance=-1pt]above:{$(A_{3,5},6) $}}] at (4,-2.4) {};
\node[circle, draw=black, fill=white, inner sep=2pt, label={[label distance=-1pt]above:{$(L_2(11),3) $}}] at (12,-6.4) {};
\node[circle, fill, inner sep=2pt, label={[label distance=-1pt]above:{$(L_2(7),2)^+ $}}] at (3,-4.8) {};
\node[circle, fill, inner sep=2pt, label={[label distance=-1pt]above:{$(L_2(7),2)^- $}}] at (5.3,-4.8) {};
\node[circle, fill, inner sep=2pt, label={[label distance=-1pt]above:{$(M_9,3) $}}] at (4.15,-5.6) {};

\end{tikzpicture}
\end{center}

Each pair of open nodes with the same label marked by $\dagger$ represents a single cubic fourfold; the node is repeated only to keep the diagram readable.

For completeness let us give the proof of Theorem~\ref{thm:19_class}.

\begin{proof}[Proof of Theorem~\ref{thm:19_class}]
    From \cite[Theorem 1.2]{laza2022automorphisms}, $\Aut^s(X)$ is isomorphic to one of the listed groups.
    \begin{enumerate}
        \item If $\Aut^s(X)\cong 3^{1+4}:2.2$, Proposition~\ref{prop:3^1+4:2.2} gives the result.
        \item If $\Aut^s(X)\cong A_6$, \S~\ref{subsect:A_6} gives the result.
        \item If $\Aut^s(X)\cong L_2(7)$, \cite[Theorem 1.2]{he2025cubicfourfoldsorder7automorphism} gives the result.
        \item If $\Aut^s(X)\cong S_5$, \S~\ref{subsection:rank19S5}  gives the result.
        \item If $\Aut^s(X)\cong M_9$, \S~\ref{subsect:M_9} gives the result.
        \item If $\Aut^s(X)\cong N_{72}$, \S~\ref{subsection:thecaseN72} gives the result.
        \item If $\Aut^s(X)\cong T_{48}$, the group-theoretic bound is $1$, thus $\Aut(X)$ must be $T_{48}$ as well. In fact, $T_{48}$ is a maximal automorphism group.
    \end{enumerate}
\end{proof}

\section{The Trivial Symplectic Automorphism Group}
\label{sect:small_rank}

The maximal abelian groups acting faithfully on smooth cubic fourfolds are classified in \cite{mayanskiy2013abelianautomorphismgroupscubic} (see also \cite{zheng2021liftableabelian} for a much shorter proof for the liftable case). In fact, we have the following result:
\begin{prop}\label{prop:32&48}
    Suppose that $X$ is a smooth cubic fourfold whose automorphism group contains an abelian subgroup of order $32$ or $48$. Then $X$ is linearly equivalent to one of the following four smooth cubic fourfolds:
    \begin{longtable}{|c|c|c|}
        \hline
            Defining Polynomial & $\Aut^s(X)$ & $\Aut(X)$ \\
            \hline
            $x_1^3+x_1x_2^2+x_2x_3^2+x_3x_4^2+x_4x_5^2+x_5x_6^2$  & $1$ & $32$\\
            \hline
            $x_1^3+x_2^3+x_2x_3^2+x_3x_4^2+x_4x_5^2+x_5x_6^2$  & $1$ & $48$\\ 
            \hline
            $x_1^3+x_1x_2^2+x_3^3+x_3x_4^2+x_4x_5^2+x_5x_6^2$ & $S_3$ & $S_3\times 24$\\
            \hline
            $x_1^3+x_1x_2^2+x_2x_3^2+x_4^3+x_4x_5^2+x_5x_6^2$ & $|\Aut^s(X)|=1944$ & $|\Aut(X)|=7776$\\
            \hline
        \caption{Smooth cubic fourfolds whose automorphism groups contain an abelian subgroup of order $32$ or $48$}
        \label{table:32&48}
    \end{longtable}
\end{prop}
\begin{proof}
    Firstly, by \cite[Corollary 1, Theorems 3 and 4]{mayanskiy2013abelianautomorphismgroupscubic}, smooth cubic fourfolds whose automorphism group contains an abelian subgroup of order $32$ or $48$ must be linearly equivalent to one of the above $4$ smooth cubic fourfolds. We just have to determine their (symplectic) automorphism groups.
    
    The first and the second are exactly $X_8^{\prime}$ and $X_5^{\prime}$ given in \cite[Section~6.1]{yang2024automorphism}, and it is shown that
\[
(\Aut^s(X_8^{\prime}),\Aut(X_8^{\prime}))=(1,32), (\Aut^s(X_5^{\prime}),\Aut(X_5^{\prime}))=(1,48).
\]

    The third is linearly equivalent to $X_3^{\prime}$ in \cite[Section~6.1]{yang2024automorphism}. It is shown there that $\Aut^s(X_3')\cong S_3$ and $\Aut(X_3')\cong 8\times(3^2:2)\cong S_3\times 24$.

    The fourth cubic fourfold is the second case of Proposition~\ref{prop:3^1+4:2.2}; see also $X_7'$ in \cite[Section~6.1]{yang2024automorphism}.
\end{proof}

In particular, we have

\begin{prop}\label{prop:sym_trivial}
    If $\Aut^s(X)$ is trivial, then all possible $\Aut(X)$ are those cyclic groups with order dividing $32$ or $48$.
\end{prop}
\begin{proof}
    Proposition~\ref{prop:32&48} gives smooth cubic fourfolds with non-symplectic indices $32$ and $48$.

    Corollary~\ref{cor:reduce} shows that all divisors of $32$ and $48$, except $1$ and $3$, are admissible for $G$. A generic smooth cubic fourfold has trivial automorphism group, so the group of order $1$ also occurs.

    The pair $(\Aut^{s}(X),\Aut(X))=(1,3)$ is also realized. Let $X$ be defined by $F=L_{3}(x_{0},x_{1},x_{2},x_{3},x_{4})+x_{5}^{3}$, where $L_{3}$ is a generic homogeneous cubic. Then $X$ is smooth and admits the non-symplectic automorphism $\frac{1}{3}(0,0,0,0,0,1)$ of order $3$. By \cite[Proposition 3.11]{yang2024automorphism}, we have $\Aut^s(X)=1$ and $\Aut(X)=3$.

    Conversely, if $\Aut^s(X)$ is trivial, then $\Aut(X)$ is cyclic. Its
    order has the form given by Proposition~\ref{prop:2k3l} and divides one
    of $21$, $30$, $32$, $33$, $36$, or $48$ by
    \cite[Proposition~3.5]{Zheng2020OnAA}. It therefore divides $32$ or
    $48$. This proves the assertion.
\end{proof}

\section{Main Table}\label{sect:main_table}
The following table includes known indices and other related information.

{\footnotesize
\begin{longtable}{|c|c|c|p{.07\textwidth}|p{.11\textwidth}|p{.065\textwidth}|p{.067\textwidth}|c|c|c|}
\hline
$\rank(S)$ & $G=\Aut^s$ & $\ord(G)$ & liftable or not & all possible indices & generic index & group-theoretic bounds & $q_S$ & $q_T$ & $i_S$\\\hline
$20$ & $3^4:A_6$ & $29160$ & No & $6$ & $6$ & $6$ & $3^{+2}9^{-1}$ & $3^{+1}9^{+1}$ & $1$\\\hline
\multirow{2}{*}{$20$} & \multirow{2}{*}{$A_7$} & \multirow{2}{*}{$2520$} & No & $1$ & $1$ & \multirow{2}{*}{$2$} & \multirow{2}{*}{$3^{-1}5^{+1}7^{-1}$} & $3^{-2}5^{+1}7^{+1}$ & \multirow{2}{*}{$1$}\\\cline{4-6}\cline{9-9}
 &  &  & Yes & $2$ & $2$ &  &  & $5^{+1}7^{+1}$ & \\\hline
$20$ & $3^{1+4}:2.2^2$ & $1944$ & No & $4$ & $4$ & $4$ & $2_6^{+2}3^{-3}$ & $2_2^{+2}3^{+2}$ & $1$\\\hline
\multirow{2}{*}{$20$} & \multirow{2}{*}{$M_{10}$} & \multirow{2}{*}{$720$} & No & $1$ & $1$ & \multirow{2}{*}{$1$} & \multirow{2}{*}{$2_3^{+1}4_7^{+1}3^{+1}5^{+1}$} & $2_5^{-1} 4_1^{+1} 3^{+2}5^{+1}$ & \multirow{2}{*}{$1$}\\\cline{4-6}\cline{9-9}
 &  &  & No & $1$ & $1$ &  &  & $2_5^{-1} 4_1^{+1} 3^{+2}5^{+1}$ & \\\hline
$20$ & $L_{2}(11)$ & $660$ & Yes & $3$ & $3$ & $3$ & $11^{+2}$ & $3^{-1}11^{+2}$ & $1$\\\hline
$20$ & $A_{3,5}$ & $360$ & Yes & $6$ & $6$ & $6$ & $3^{-2}5^{-2}$ & $3^{-1}5^{-2}$ & $1$\\\hline
$19$ & $3^{1+4}:2.2$ & $972$ & No & $2$ & $2$ & $6$ & $2_7^{+1}3^{-4}$ & $2_1^{+1}3^{-3}$ & $1$\\\hline
\multirow{2}{*}{$19$} & \multirow{2}{*}{$A_{6}$} & \multirow{2}{*}{$360$} & No & $1$ & $1$ & \multirow{2}{*}{$2$} & \multirow{2}{*}{$4_3^{-1}3^{+2}5^{+1}$} & $4_5^{-1}3^{-3}5^{+1}$ & \multirow{2}{*}{$1$}\\\cline{4-6}\cline{9-9}
 &  &  & Yes & $2$ & $2$ &  &  & $4_5^{-1}3^{+1}5^{+1}$ & \\\hline
$19$ & $L_{2}(7)$ & $168$ & Yes & $1$, $2$ & $1$ & $2$ & $4_7^{+1}7^{+2}$ & $4_1^{+1}3^{-1}7^{+2}$ & $1$\\\hline
\multirow{2}{*}{$19$} & \multirow{2}{*}{$S_{5}$} & \multirow{2}{*}{$120$} & Yes & $1$ & $1$ & \multirow{2}{*}{$6$} & \multirow{2}{*}{$4_5^{-1}3^{-1}5^{-2}$} & $4_3^{-1}3^{-2}5^{-2}$ & \multirow{2}{*}{$1$}\\\cline{4-6}\cline{9-9}
 &  &  & Yes & $2$ & $2$ &  &  & $4_3^{-1}5^{-2}$ & \\\hline
$19$ & $M_{9}$ & $72$ & No & $1$, $3$ & $1$ & $3$ & $2_5^{+3}3^{+1}9^{+1}$ & $2_3^{+3} 3^{+2} 9^{-1}$ & $1$\\\hline
$19$ & $N_{72}$ & $72$ & Yes & $2$ & $2$ & $6$ & $4_7^{+1}3^{+2}9^{+1}$ & $4_1^{+1}3^{+1}9^{-1}$ & $1$\\\hline
$19$ & $T_{48}$ & $48$ & Yes & $1$ & $1$ & $1$ & $2_1^{+1}8_{\mathrm{II}}^{-2}3^{+1}$ & $2_7^{+1}8_{\mathrm{II}}^{-2}3^{+2}$ & $1$\\\hline
$18$ & $3^{1+4}:2$ & $486$ & No & $2$, $4$, $6$, $12$ & $2$ & $12$ & $3^{-5}$ & $3^{+4}$ & $1$\\\hline
\multirow{2}{*}{$18$} & \multirow{2}{*}{$A_{4,3}$} & \multirow{2}{*}{$72$} & No &  & $1$ & \multirow{2}{*}{$6$} & \multirow{2}{*}{$4_{\mathrm{II}}^{-2}3^{+3}$} & $4_{\mathrm{II}}^{-2}3^{+4}$ & \multirow{2}{*}{$1$}\\\cline{4-6}\cline{9-9}
 &  &  & Yes &  & $2$ &  &  & $4_{\mathrm{II}}^{-2}3^{-2}$ & \\\hline
\multirow{2}{*}{$18$} & \multirow{2}{*}{$A_{5}$} & \multirow{2}{*}{$60$} & Yes &  & $1$ & \multirow{2}{*}{$6$} & \multirow{2}{*}{$2_{\mathrm{II}}^{-2}3^{-1}5^{-2}$} & $2_{\mathrm{II}}^{-2}3^{-2}5^{-2}$ & \multirow{2}{*}{$1$}\\\cline{4-6}\cline{9-9}
 &  &  & Yes &  & $2$ &  &  & $2_{\mathrm{II}}^{-2}5^{-2}$ & \\\hline
\multirow{2}{*}{$18$} & \multirow{2}{*}{$3^{2}.4$} & \multirow{2}{*}{$36$} & No &  & $1$ & \multirow{2}{*}{$6$} & \multirow{2}{*}{$2_6^{+2}3^{+2}9^{+1}$} & $2_2^{+2}3^{-3}9^{-1}$ & \multirow{2}{*}{$1$}\\\cline{4-6}\cline{9-9}
 &  &  & Yes &  & $2$ &  &  & $2_2^{+2}3^{+1}9^{-1}$ & \\\hline
\multirow{2}{*}{$18$} & \multirow{2}{*}{$S_{3,3}$} & \multirow{2}{*}{$36$} & Yes &  & $2$ & \multirow{2}{*}{$6$} & \multirow{2}{*}{$2_{\mathrm{II}}^{-2}3^{-3}9^{+1}$} & \multirow{2}{*}{$2_{\mathrm{II}}^{-2}3^{+2}9^{-1}$} & \multirow{2}{*}{$3$}\\\cline{4-6}
 &  &  & Yes &  & $2$ &  &  &  & \\\hline
$18$ & $F_{21}$ & $21$ & Yes & $1$, $2$, $6$ & $1$ & $6$ & $7^{-3}$ & $7^{+3}3^{-1}$ & $1$\\\hline
$18$ & $\mathrm{Hol}(5)$ & $20$ & Yes &  & $1$ & $6$ & $2_6^{+2}5^{+3}$ & $2_2^{+2}3^{-1}5^{+3}$ & $1$\\\hline
$18$ & $\mathrm{QD}_{16}$ & $16$ & Yes & $1$, $2$ & $1$ & $2$ & $2_1^{+1}4_1^{+1}8_{\mathrm{II}}^{+2}$ & $2_7^{+1}4_7^{+1}8_{\mathrm{II}}^{+2}3^{-1}$ & $1$\\\hline
\multirow{2}{*}{$17$} & \multirow{2}{*}{$S_{4}$} & \multirow{2}{*}{$24$} & Yes &  & $1$ & \multirow{2}{*}{$6$} & \multirow{2}{*}{$4_5^{+3}3^{+2}$} & $4_3^{+3}3^{-3}$ & \multirow{2}{*}{$1$}\\\cline{4-6}\cline{9-9}
 &  &  & Yes &  & $2$ &  &  & $4_3^{+3}3^{+1}$ & \\\hline
$17$ & $Q_{8}$ & $8$ & Yes &  & $1$ & $3$, $4$ & $2_5^{+3}8_{\mathrm{II}}^{-2}$ & $2_3^{+3}8_{\mathrm{II}}^{-2}3^{-1}$ & $1$\\\hline
\multirow{2}{*}{$16$} & \multirow{2}{*}{$A_{3,3}$} & \multirow{2}{*}{$18$} & No &  & $1$ & \multirow{2}{*}{$12$} & \multirow{2}{*}{$3^{+4}9^{+1}$} & $3^{-5}9^{-1}$ & \multirow{2}{*}{$3$}\\\cline{4-6}\cline{9-9}
 &  &  & Yes &  & $2$ &  &  & $3^{+3}9^{-1}$ & \\\hline
\multirow{2}{*}{$16$} & \multirow{2}{*}{$D_{12}$} & \multirow{2}{*}{$12$} & Yes &  & $1$ & \multirow{2}{*}{$12$} & \multirow{2}{*}{$2_\mathrm{II}^{+4}3^{+4}$} & $2_\mathrm{II}^{+4}3^{-5}$ & \multirow{2}{*}{$1$}\\\cline{4-6}\cline{9-9}
 &  &  & Yes &  & $2$ &  &  & $2_\mathrm{II}^{+4}3^{+3}$ & \\\hline
\multirow{2}{*}{$16$} & \multirow{2}{*}{$A_{4}$} & \multirow{2}{*}{$12$} & Yes &  & $1$ & \multirow{2}{*}{$6$} & \multirow{2}{*}{$2_\mathrm{II}^{-2}4_\mathrm{II}^{-2}3^{+2}$} & $2_\mathrm{II}^{-2}4_\mathrm{II}^{-2}3^{-3}$ & \multirow{2}{*}{$1$}\\\cline{4-6}\cline{9-9}
 &  &  & Yes &  & $2$ &  &  & $2_\mathrm{II}^{-2}4_\mathrm{II}^{-2}3^{+1}$ & \\\hline
$16$ & $D_{10}$ & $10$ & Yes &  & $1$ & $12$ & $5^{+4}$ & $3^{-1}5^{+4}$ & $1$\\\hline
$15$ & $D_{8}$ & $8$ & Yes &  & $1$ & $4$, $6$ & $4_7^{+5}$ & $4_1^{+5}3^{-1}$ & $1$\\\hline
$14$ & $4$ & $4$ & Yes &  & $1$ & $8$, $12$ & $2_6^{+2}4_\mathrm{II}^{+4}$ & $2_2^{+2}4_\mathrm{II}^{+4}3^{-1}$ & $1$\\\hline
\multirow{2}{*}{$14$} & \multirow{2}{*}{$S_3$} & \multirow{2}{*}{$6$} & Yes &  & $1$ & \multirow{2}{*}{$24$} & \multirow{2}{*}{$2_{\mathrm{II}}^{-2}3^{-5}$} & $2_{\mathrm{II}}^{-2}3^{-6}$ & \multirow{2}{*}{$1$}\\\cline{4-6}\cline{9-9}
 &  &  & Yes & $2$, $4$, $6$, $8$, $12$, $24$ & $2$ &  &  & $2_{\mathrm{II}}^{-2}3^{+4}$ & \\\hline
$12$ & $2^{2}$ & $4$ & Yes & & $1$ & $12$ & $2_\mathrm{II}^{-6}4_\mathrm{II}^{-2}$ & $2_\mathrm{II}^{-6}4_\mathrm{II}^{-2}3^{-1}$ & $40$\\\hline
\multirow{2}{*}{$12$} & \multirow{2}{*}{$3$} & \multirow{2}{*}{$3$} & Yes &  & $1$ & \multirow{2}{*}{$16$, $24$} & \multirow{2}{*}{$3^{+6}$} & $3^{-7}$ & \multirow{2}{*}{$1$}\\\cline{4-6}\cline{9-9}
 &  &  & Yes &  & $2$ &  &  & $3^{+5}$ & \\\hline
$8$ & $2$ & $2$ & Yes &  & $1$ & $16$, $24$ & $2_\mathrm{II}^{+8}$ & $2_\mathrm{II}^{+8}3^{-1}$ & $1$\\\hline
$0$ & $1$ & $1$ & Yes & $1$, $2$, $3$, $4$, $6$, $8$, $12$, $16$, $24$, $32$, $48$ & $1$ & $32$, $48$ & $1$ & $3^{-1}$ & $1$\\\hline
\caption{Known Indices and Other Information}
\label{table: main}
\end{longtable}
}

\begin{rmk}\label{rem:table}
    We explain some settings of Table~\ref{table: main}.
    \begin{enumerate}
        \item For a fixed group $G$, multiple rows correspond to distinct irreducible moduli spaces of smooth cubic fourfolds with generic symplectic automorphism group $G$. There are $48$ such families in total, including non-isomorphic cubic fourfolds with the same maximal automorphism group. Those moduli spaces are given in \cite{KOIKE202512}.
        \item The columns ``generic index'' and ``group-theoretic bounds'' give the bounds that all possible indices must be divisible by the generic index and divide one of the group-theoretic bounds.
        \item For entries in the column ``all possible indices'' that we leave blank, we only know the generic indices and group-theoretic bounds.
        \item The number $i_S$ is the index of the image of $O(S)\to O(q_S)$ in $O(q_S)$. The values $q_S$ and $i_S$ are from \cite[Table~1]{hohn2016290}. In fact, we do not really use $i_S$ since we can always directly find the global non-symplectic automorphism in the families.
        \item The column ``liftable'' records the liftability of the given projective representation of $G$; see also \cite{KOIKE202512}. A subgroup $G\subset \PGL(6,\CC)$ is \emph{liftable} if there exists a subgroup of $\GL(6,\CC)$ that maps isomorphically onto $G$ under the natural projection $\GL(6,\CC)\to\PGL(6,\CC)$.
        \item If $\rank(S)=20$, the moduli spaces are discrete points, so the possible indices and generic indices coincide. 
    \end{enumerate}
\end{rmk}

As a combination of our results, we could give the proof of Theorem~\ref{thm:main}:

\begin{proof}[Proof of Theorem~\ref{thm:main}] Known indices are deduced via various ways. Generally one can apply Proposition~\ref{prop:2k3l}, Corollary~\ref{cor:reduce}, and the group-theoretic bounds.
    \begin{enumerate}
        \item The cases where $G$ is maximal are already classified in \cite[Theorem 1.8]{laza2022automorphisms}.
        \item The cases $\rank(S)=19$ are classified in Theorem~\ref{thm:19_class} and explained in \S~\ref{sect:rk19}.
        \item The cases where $7\mid\ord(G)$ are already classified in \cite[Theorem 1.2]{he2025cubicfourfoldsorder7automorphism}.
        \item For $\QD_{16}$, the extra non-symplectic automorphism can be written as $\frac{1}{2}(1,0,0,0,0,0)$, which does not act on the whole family. See \cite[Theorem 1.2(7)(c)]{laza2022automorphisms}.
        \item For $S_3$, see $X_3^{\prime}$ in \cite[Section~6.1]{yang2024automorphism}. By calculating the character of $\Aut^s(X_3^{\prime})\cong S_3$, one can check that it belongs to the second family and has non-symplectic index $24$. Corollary~\ref{cor:reduce} then gives the indices $2$, $4$, $6$, $8$, $12$, and $24$.
        \item For $3^{1+4}:2$, $(12)$ acts non-symplectically on the whole family; see \cite[Section~6.1(i)]{KOIKE202512}. Index $12$ can be realized by $X_2^{\prime}$ in \cite[Section~6.1]{yang2024automorphism}. Corollary~\ref{cor:reduce} then gives the indices $2$, $4$, $6$, and $12$.
        \item The case $G$ is trivial is explained in Proposition~\ref{prop:sym_trivial}.
    \end{enumerate}
\end{proof}

\bibliography{reference}
\Addresses

\end{document}